\documentclass[a4paper,11pt]{article}
\usepackage{amsmath,amssymb,amsthm}
\usepackage[top=2.5cm,bottom=2.5cm,left=2.5cm,right=2.5cm]{geometry}
\usepackage{enumitem}
\usepackage{mathrsfs}
\usepackage{graphicx}
\usepackage{xcolor}
\usepackage{graphicx,float,epstopdf}
\usepackage{hyperref}
\usepackage[T1]{fontenc}
\usepackage[english]{babel}
\usepackage{dsfont}
\usepackage{enumitem}
\usepackage{footmisc}

\newtheorem{theorem}{Theorem}[section]
\newtheorem{lemma}[theorem]{Lemma}
\newtheorem{corollary}[theorem]{Corollary}

\newtheorem{remark}[theorem]{Remark}
\newtheorem{example}[theorem]{Example}
\numberwithin{equation}{section}

\title{On box dimension of the graphs of the generalized Riemann-type functions}
\author{\sl Yurong Wu$^{a}$ and Guoping Zhan$^{b*}$}

\date{{\small \em $^a$School of Mathematical Sciences, Zhejiang University of Technology,\\
Hangzhou, 310023, P. R. China \\
e-mail: wuyurong2022@zjut.edu.cn; wuyurong2003@163.com\\
\ $^b$\thanks{Corresponding author: Guoping Zhan (e-mail: zhangp@zjut.edu.cn)}
School of Mathematical Sciences, Zhejiang University of Technology,\\
Hangzhou, 310023, P. R. China
\\e-mail: zhangp@zjut.edu.cn}}

\begin{document}

\maketitle
\begin{abstract}
We investigate the box dimension of the graphs of a class of continuous periodic functions
$G_\delta(x)=\sum_{n=1}^{\infty}g(n^{2}x)n^{-1-\delta}$ with 1-periodic Lipschitz functions $g$
and $0<\delta\le 1$, which generalizes the result of the classical Riemann function corresponding to
$g(x)=\sin(2\pi x)$ and $\delta=1$. More precisely, we first prove that the lower box dimension of the graph of $G_{\delta}$ is no less than $\frac74-\frac{\delta}{2}$ when the Fourier coefficients of $g$ satisfy an arithmetic non-vanishing condition related to the distribution of quadratic residues. This result is new and non-trivial even when $g$ has a finite Fourier expansion, highlighting the intrinsic arithmetic complexity of the series. Secondly, if $g'$ is Lipschitz continuous on $\mathbb{R}$, we show that the upper box dimension does not exceed \(\frac74-\frac{\delta}{2}\), which extends earlier work of Chamizo and C\'ordoba and reveals deep connection between the regularity
of $g$ and the fractal dimension of the associated Riemann-type series. In the end, we give some illustrative examples and propose
some further problems.

\end{abstract}

\medskip
\noindent\textbf{Keywords:} Riemann's function, box dimension, Farey dissection, Gauss sums, quadratic reciprocity,
Dirichlet character.

\tableofcontents

\section{Introduction}

\subsection*{Historical background}
In 1872, Karl Weierstrass first proved that $$W(x)=\sum_{n=0}^{\infty} a^n \cos(b^n\pi x)$$
called Weierstrass function later, is continuous and nowhere differentiable on $\mathbb{R}$ for
$a\in(0,1)$ and odd integer $b>1$ with $ab>1+3\pi/2$. More generally, Hardy \cite{Hardy1916} showed in 1916
that $W$ is nowhere differentiable for $a\in(0,1)$ and $b>1$ with $ab\ge1$.

\par For over a century and a half, the classical Weierstrass function
has attracted number of mathematicians. For instance, Hunt \cite{Hunt1998} showed that the graph of
$$w_\Theta(x)=\sum_{n=0}^\infty a^n\cos(2\pi(b^n x+\theta_n))$$ with i.i.d. uniform phases
$\theta_n\in\Theta=\{\theta_0,\theta_1,...\}$  has Hausdorff dimension $D=2+\frac{\log a}{\log b}$ almost surely.
The same result extends to a broad class $f_\Theta(x)=\sum_{n=0}^\infty a_n g(b_n x+\theta_n)$ with periodic function $g$
under mild conditions.

\par Determining the exact Hausdorff dimension of the graph of Weierstrass functions remains a long-standing
open problem before the landmark paper Shen \cite{Shen2018}. Detailly, Shen \cite{Shen2018} proved that the graph
of the classical Weierstrass function
$$W_{\lambda,b}(x)=\sum_{n=0}^{\infty} \lambda^n \cos(2\pi b^n\pi x)$$
has Hausdorff dimension $D=2+\frac{\log \lambda}{\log b}$ for each integer $b\ge2$ and $\lambda\in(1/b,1)$,
settling a long-standing conjecture. The same result is extended to a broad class of functions
$f_{\lambda,b}^{\phi}(x)=\sum_{n=0}^{\infty}\lambda^n\phi(b^nx)$ with non-constant $C^2$ periodic function
$\phi$, provided $\lambda b$ is sufficiently close to 1.

\par Furthermore, Ren and Shen \cite{Shen-Ren2021} made a decisive breakthrough.
They considered the general Weierstrass-type function
$$W(x)=\sum_{n=0}^{\infty} \lambda^n \phi(b^n x)$$
for a $\mathbb{Z}$-periodic real analytic function $\phi$ with integer
$b\ge2$ and $1/b<\lambda<1$. They proved the following dichotomy theorem:
either $W$ is real analytic, or its graph has Hausdorff dimension
\[\dim_H(graph(W))=2+\log_b\lambda.\]
This work extends the classical result for the cosine Weierstrass function and is proved by a combination
of Ledrappier-Young theory, Hochman's entropy increment criterion, and a novel transversality analysis,
together with a renormalization argument based on regulating periods.

\par Another fruitful direction involves the renowned Riemann's function
\begin{equation*}
F(x)=\sum_{n=1}^{\infty}\frac{\sin(2\pi n^{2}x)}{n^{2}}.
\end{equation*}
However, as Hardy \cite{Hardy1916} pointed out, the non-differentiable question concerning Riemann's function is
much more difficult than any of those connected
with Weierstrass's function, owing to the comparatively slow increase of the sequence $n^2$,
similar difficulty also exists for the study of fractal dimension.
Hardy \cite{Hardy1916} also demonstrated that, despite the initial belief of being nowhere differentiable,
the function $F$ is indeed nondifferentiable at every irrational point and at certain rational points.
Later, Gerver \cite{Gerver1970} verified differentiability at infinitely many rational points
with denominator $\equiv 2\pmod 4$.

\par Furthermore, the geometric complexity of the graph of such functions has been intensively studied.
For a continuous function $f:[0,1]\to\mathbb{R}$ denote by $graph(f)$ its graph.
The \emph{box dimension} of $graph(f)$ is defined as
\[\dim_{\mathrm{B}}\bigl(graph(f)\bigr)=\lim_{N\to\infty}\frac{\log A_N(f)}{\log N}\]
with $A_N(f)$ the number of squares of side length $1/N$ intersecting $graph(f)$,
and replace limit by upper and lower limits when necessary.

\par For instance, Jaffard \cite{Jaffard1996}
 established the multifractal nature of Riemann's function by determining pointwise H\"{o}lder spectrum of
 its singularities, and linked the H\"{o}lder exponent at irrationals to the corresponding irrationality exponent.
Subsequently, by Farey dissection Chamizo and C\'{o}rdoba \cite{Chamizo1999} showed
$\dim_{\mathrm{B}}\bigl(graph(F_{\delta})\bigr)=\frac{7}{4}-\frac{\delta}{2}$
for $$F_{\delta}(x)=\sum_{n=1}^{\infty}\frac{\sin(2\pi n^{2}x)}{n^{1+\delta}}$$ with $0<\delta<1$,
using lower bounds from quadratic residues and upper bounds from Weyl-type incomplete Gauss sums;
the cosine case is analogous.

\par After that, this subject was broadened in several directions.
Chamizo \cite{Chamizo2004} further obtained the box dimension of graphs of elliptic curves, theta functions, and modular forms by Fourier series arising as fractional integrals of automorphic forms. Later, C\'ordoba \cite{Cordoba2008} surveyed the interface between number theory and harmonic analysis, covering Gaussian trigonometric series, convergence of rearranged Fourier series, and fractal dimensions of Riemann-type graphs.
Also, Eceizabarrena \cite{Eceizabarrena2020, Eceizabarrena2021} studied the geometry of the complex Riemann function
$$\phi(t)=\sum_{k\in\mathbb{Z}}\frac{e^{-4\pi^2ik^2t}-1}{-4\pi^2k^2},$$
which arises as the trajectory of a corner of a polygonal vortex filament under the binormal flow.
The author proved in \cite{Eceizabarrena2020} that $\phi(\mathbb{R})$ has no tangent anywhere in the sense
of Falconer's geometric tangent definition, and in \cite{Eceizabarrena2021} that
 $1\le\dim_{\mathcal{H}}\phi(\mathbb{R})\le4/3$ on the Hausdorff dimension of its image,
together with geometric properties such as absence of tangents, respectively.

\par Motivated by \cite{Chamizo1999}, in the present paper we consider a natural generalization of the classical Riemann function defined by
\[G_{\delta}(x)=\sum_{n=1}^{\infty}\frac{g(n^{2}x)}{n^{1+\delta}}\]
with $0<\delta\le1$ and $1$-periodic continuous function $g:\,\mathbb{R}\to\mathbb{R}$. Under mild arithmetic conditions
on the Fourier coefficients of $g$, we obtain a sharp lower bound
for the box dimension of $graph(G_{\delta})$. Specifically, let $g(x)=\sum_{k\in\mathbb{Z}}C_ke^{2\pi i kx}$
with $\sum_{k=1}^{\infty}|C_k|k^{\frac{\varepsilon+\delta}{2}}<\infty$ for some $\varepsilon>0$.
If there exists a positive square-free integer $d_0$ such that
\begin{equation}\label{nonvanishing}\sum_{\substack{k=m^{2}d_0 \\ m\in\mathbb{N}^+}}C_k k^{\delta/2}\neq0,\end{equation}
then
\[\underline{\dim}_{\mathrm{B}}\bigl(graph(G_{\delta})\bigr)\ge\frac{7}{4}-\frac{\delta}{2}.\]

\par Moreover, if $g'$ is Lipschitz continuous, then
\[\overline{\dim}_{\mathrm{B}}\bigl(graph(G_{\delta})\bigr)\le\frac{7}{4}-\frac{\delta}{2}.\]
Consequently, under the non-vanishing condition (\ref{nonvanishing}), the exact box dimension is
\[\dim_{\mathrm{B}}\bigl(graph(G_{\delta})\bigr)=\frac{7}{4}-\frac{\delta}{2}\]
for $0<\delta\le1$. This extends the result of \cite{Chamizo1999}
and establishes a direct link between the regularity of the seed
function $g$ and the fractal dimension of the associated Riemann-type series.

\par The novelty of our method is an arithmetic mechanism for the lower bound, based on decomposing the function at rational points with prime denominators and employing the theory of quadratic residue and Gauss sum. A central ingredient is Lemma 3.3, which, under the non-vanishing condition, ensures infinitely many primes for which a weighted sum of Legendre symbol is bounded away from zero; its proof uses quadratic reciprocity and careful analysis of square-free integers. For the upper bound, we use a Farey partition of order $N^{1/2}$, together with Weyl-type estimates for incomplete sums of $g(n^2x)$ near rationals.

\par Finally, we address the critical case in which the weighted sums vanish for every square-free integer $d_0$.
Examples including finite trigonometric polynomials, functions involving modular form, and
$L$-functions illustrate that the vanishing condition can lead either to a collapse of the dimension or,
in the presence of extra symmetries, to a more subtle behaviour. These observations hint at connections with
the Prime Number Theorem and analytic properties of $L$-functions, and suggest further work at the interface
between number theory and fractal geometry.

\subsection*{Main results}
In this paper we establish sharp lower and upper bounds for the box dimension of \(graph(G_{\delta})\)
under mild arithmetic and decay conditions on the Fourier coefficients of \(g\).
\begin{theorem}\label{thm:lower}
Let $g(x)=\sum_{k\in\mathbb{Z}}C_ke^{2\pi i k x}$ with $\sum_{k=1}^{\infty}|C_k|k^{\frac{\varepsilon+\delta}{2}}<\infty$
for some $\varepsilon>0$. If there exists a positive square-free
integer $d_0$ such that
\[\sum_{\substack{k=m^{2}d_0\\ m\in\mathbb{N}^+}}C_k\,k^{\delta/2}\neq0,\]
then
\[\underline{\dim}_{\mathrm{B}}\bigl(graph(G_\delta)\bigr)\ge\frac74-\frac{\delta}{2}.\]
\end{theorem}

\begin{remark}
{\rm The non-vanishing condition {\rm(\ref{nonvanishing})} holds for many natural choices of $g$.
For example, for any nonzero $C_1=a+bi\in\mathbb{C}$, define $g(t)=C_1e^{2\pi it}+C_{-1}e^{-2\pi it}
=2a\cos(2\pi t)-2b\sin(2\pi t)$. Then its Fourier coefficients satisfy {\rm(\ref{nonvanishing})}
with $d_0=1$ for every $\delta>0$. Thus Theorem \ref{thm:lower} recovers the lower bound
$\dim_B\ge5/4$ for the classical Riemann function.}
\end{remark}

\begin{theorem}\label{thm:upper}
Let \(g:\mathbb{R}\to\mathbb{R}\) be a {\rm 1}-periodic function. If $g'$ is Lipschitz continuous, then
\[\overline{\dim}_{\mathrm{B}}\bigl(graph(G_\delta)\bigr)\le\frac74-\frac{\delta}{2}.\]
\end{theorem}

\begin{remark} {\rm For two continuous functions $f_1$ and $f_2$, by the definition of upper box dimension it is easy to see
\[\overline{\dim}_{\mathrm{B}}\bigl(graph(f_1+f_2)\bigr)
\leq\max\{\overline{\dim}_{\mathrm{B}}\bigl(graph(f_1)\bigr),\
\overline{\dim}_{\mathrm{B}}\bigl(graph(f_2)\bigl)\}.\]
So Theorem \ref{thm:upper} is nontrivial only
in the case when the seed function $g$ has infinitely many nonzero Fourier coefficients.}
\end{remark}

\begin{corollary}\label{dimension-formula}
Let \(g:\mathbb{R}\to\mathbb{R}\) be \(1\)-periodic with Fourier expansion \(g(x)=\sum_{k\in \mathbb{Z}}C_ke^{2\pi ikx}\)
and $g'$ is Lipschitz continuous. If there exists a positive square-free integer \(d_0\) such that {\rm(\ref{nonvanishing})}
holds, then for every \(0<\delta\le1\) we have
\[\dim_{\mathrm{B}}\bigl(graph(G_\delta)\bigr)=\frac74-\frac{\delta}{2}.\]
In particular, the above dimension formula holds for any trigonometric polynomial
$g(x)=\sum_{k=-N}^{N}C_ke^{2\pi i k x}$ which satisfies
$\sum_{1\leq m^{2}d_0\leq N}C_{m^2 d_0}\,(m^2 d_0)^{\delta/2}\neq0$
for some positive square-free integer \(d_0\).
\end{corollary}

\begin{remark}
\rm {Even for smooth finite trigonometric polynomials, the dimension formula of Corollary \ref{dimension-formula}
is nontrivial, as $G_{\delta}$ acquires fractal structure from square-frequency arithmetic. The dimension depends not
only on the regularity of $g$, but on the interplay of its Fourier coefficients with quadratic Gauss sums.
Failure of (\ref{nonvanishing}) may lower the dimension (see Section \ref{sec:critical}),
showing that the $\frac74-\frac{\delta}{2}$ scaling relies on a delicate arithmetic non-cancellation.}
\end{remark}

\subsection*{Structure of the paper}
The remainder of this paper is organised as follows.
We begin in Section \ref{sec:prelim} by collecting some preliminary lemmas from fractal geometry,
analytic number theory and harmonic analysis, which will serve as the foundational tools for our arguments.
In Section \ref{sec:lower}, we establish the lower bound on the lower box dimension of the graph of $G_\delta$
by technical analysis of the mean variation of the differences $\Delta G_\delta(a/p)$ over the
set of quadratic residues. Section \ref{sec:upper} is devoted to establishing the upper bound on the upper box
dimension of the graph of $G_\delta$ under Lipschitz
 condition on $g'$. The critical case is addressed in Section \ref{sec:critical}, where we discuss the potential
 complexity that may arise and provide several illustrative examples. Finally, Section \ref{sec:conclusion} proposes
 some relevant problems for further investigation.

\section{Preliminaries}\label{sec:prelim}
Throughout this article, for two positive quantities $A$ and $B$, $A\lesssim B$ means that there exists an absolute constant $C>0$
such that $A\le CB$, and $|A|\lesssim|B|$ is denoted by $A=O(B)$. Similarly, $A\lesssim_{\varepsilon}B$ signifies there exists
a constant $C_{\varepsilon}>0$ depending only on the parameter $\varepsilon$ such that $A\le C_{\varepsilon}B$. Also, $A\asymp B$
 suggests $C_1\le A\le C_2B$ with two universal constants $C_1,\, C_2>0$.
\par For readers' convenience, some preliminary lemmas of necessary tools are given as follows.
\subsection{Box dimension}
Let \(f:[0,1]\to\mathbb{R}\) be a continuous function. Its $graph$ is defined as the set $$graph(f)=\{(x,f(x)):x\in[0,1]\}.$$
For a given positive integer $N$, define an $N\times N$ grid $\mathcal{M}_N$ on the unit square $[0,1]^2$ as
$$\mathcal{M}_N=\left\{x=\frac{i}{N},y=\frac{j}{N}:\ 0\le i,j\le N\right\}.$$
The number of grid squares of \(\mathcal{M}_N\)
that intersect $graph(f)$ is denoted by $A_N(f)$.
\par Following (2.5) and (2.6) in \cite[Chapter 2]{Falconer2014},
the \emph{upper} and \emph{lower box dimension}
can be equivalently defined respectively as \[
\overline{\dim}_{\mathrm{B}}\bigl(graph(f)\bigr)=\limsup_{N\to\infty}\frac{\log A_N(f)}{\log N},\qquad
\underline{\dim}_{\mathrm{B}}\bigl(graph(f)\bigr)=\liminf_{N\to\infty}\frac{\log A_N(f)}{\log N}.\]
When these two limits coincide, their common value is called \emph{box dimension}, which is denoted by
$\dim_{\mathrm{B}}\bigl(graph(f)\bigr)$.
\par Let $\{[j/N,(j+1)/N]\}_{j=0}^{N-1}$ be a partition of
$[0,1]$. For each $j$, denote the $oscillation$ of $f$ on subinterval as $${\rm osc}_j(f)=\sup_{x, y\in[j/N,(j+1)/N]}|f(x)-f(y)|.$$
From the definition of $A_N(f)$, it is easy to see that \begin{equation} \label{ANf}
A_N(f)\leq \sum_{j=0}^{N-1}\big(1+N\cdot{\rm osc}_j(f)\big),\end{equation}which is a central tool for estimating $A_N(f)$.

\subsection{Farey dissection}
Following \cite[Chapter III ]{HardyWright2008}, given a large integer \(N\), the $Farey\ sequence$ of order \(N^{1/2}\) consists of all reduced fractions \(a/p\) with \(0\le a\le p\le N^{1/2}\). The $Farey\ interval$ corresponding to each such $a/p$ is denoted by
\[J_{a/p}=\left(\frac{a'+a}{p'+p},\,\frac{a+a''}{p+p''}\right],\]
where \(\frac{a'}{p'}<\frac{a}{p}<\frac{a''}{p''}\) are consecutive Farey fractions, form a dissection
of \((0,1]\). We claim for $J_{a/p}$ that
\begin{equation}\label{Farey-inclusion} J_{a/p}\subset\Bigl\{x\in(0,1]:\,\Bigl|x-\frac{a}{p}\Bigr|<\frac1{p N^{1/2}}\Bigr\}.\end{equation}
\par Indeed, by \cite[Theorem28 and 30]{HardyWright2008} we obtain
\[ap'-a'p=1, \qquad a''p-ap''=1\] and \[p+p'>N^{1/2},\qquad p+p''>N^{1/2}.\]
This imply \[\frac{a}{p}-\frac{a'+a}{p'+p}=\frac{a(p'+p)-p(a'+a)}{p(p'+p)}=\frac{1}{p(p'+p)}<\frac{1}{p N^{1/2}}\]
and
\[\frac{a+a''}{p+p''}-\frac{a}{p}=\frac{p(a+a'')-a(p+p'')}{p(p+p'')}=\frac{1}{p(p+p'')}<\frac{1}{p N^{1/2}}.\]
Hence, the inclusion (\ref{Farey-inclusion}) holds.
\par Farey dissection will be adapted to estimate trigonometric sums with quadratic frequencies.

\subsection{Legendre's symbol, Gauss sum and Dirichlet Character}
Following \cite{Apostol1976}, for an odd prime \(p\) and integer \(k\) with \(p\nmid k\), if the congruence
$x^2\equiv k({\rm mod}\ p)$ has a solution, we say that $k$ is a $quadratic\ residue\ {\rm mod}\ p$ and write $kRp$; otherwise,
$k$ is a $quadratic\ nonresidue\ {\rm mod}\ p$ and write $k\overline{R}p$. Legendre's symbol $\left(\frac{k}{p}\right)$ is defined as
\[\left(\frac{k}{p}\right)=\begin{cases}
1,&{\rm if}\ k R p,\\
-1,&{\rm if}\ k \overline{R}p,\\
0,&{\rm if}\ p\mid k.
\end{cases}
\]
For integers $n>0$ and $k$, the sum $$G(k;n)=\sum\limits_{r=1}^n e^{2\pi ikr^2/n}$$
is called a $quadratic\ Gauss\ sum$.
\par Combining Theorems 9.4, 9.5 and 9.8 in \cite[Chapter 9]{Apostol1976},
together with (28)-(30) in \cite[Section 9.10]{Apostol1976}, derives the following lemma,
which will be used to establish Lemma \ref{lem:leg-same} and prove Theorem \ref{thm:lower}.
\begin{lemma}\label{lem:leg-law}{\rm (i)} If \(p\) is an odd prime, then
\[ \left(\frac{-1}{p}\right)=(-1)^{\frac{p-1}{2}},\qquad \left(\frac{2}{p}\right)=(-1)^{\frac{p^{2}-1}{8}}.\]
{\rm (ii)(Quadratic reciprocity law)} If $p$ and $q$ are distinct odd primes, then
\[\left(\frac{q}{p}\right)=\left(\frac{p}{q}\right)(-1)^{\frac{(p-1)(q-1)}{4}}.\]
{\rm (iii)} Given an odd prime \(p\) and integer \(k\) with \(p\nmid k\), we have
\[ G(k;p)=\sum_{r=0}^{p-1}e^{2\pi i k r^{2}/p}= \epsilon_p\sqrt{p}\,\left(\frac{k}{p}\right),
\qquad
\epsilon_p=\begin{cases}
1,&p\equiv1\pmod4,\\[2pt]
\ i,&p\equiv3\pmod4.
\end{cases}
\]
\end{lemma}
\par  Set \(R(p):=\{a\bmod p:\left(\frac{a}{p}\right)=1\}\), which consists of quadratic residues ${\rm mod}\ p$.
For each integer $n$ with $p\nmid n$, one can check that \begin{equation} \label{sum-Rp}
\sum_{a\in R(p)}e^{2\pi i k n^2a/p}=\frac12\Bigl(\epsilon_p\sqrt{p}\,\left(\frac{k}{p}\right)-1\Bigr).
\end{equation}
Actually, we consider the map \(\varphi : R(p) \longrightarrow R(p)\) defined by \(\varphi(a)=n^2a\pmod{p}\) for each $a\in R(p)$.
Note that $p\nmid n$ guarantees $\varphi$ is a permutation of \(R(p)\). This implies
\[\sum_{a\in R(p)}e^{2\pi ik n^2a/p}=\sum_{a\in R(p)}e^{2\pi ik\varphi(a)/p}=\sum_{b\in R(p)}e^{2\pi ikb/p}.
\]

\par Also, for each \(b\in R(p)\) the congruence \(r^2 \equiv b \pmod{p}\) has exactly two solutions in $\{1,2,...,p-1\}$, then
\[G(k;p)=1+2\sum_{b\in R(p)}e^{2\pi ikb/p}.\]
Therefore, by Lemma \ref{lem:leg-law}(iii) we have
\[\sum_{a\in R(p)} e^{2\pi i k n^2a/p}=\sum_{b \in R(p)} e^{2\pi ikb/p}=\frac{1}{2}\big(G(k;p)-1\big)=\frac{1}{2}\left(\varepsilon_p\sqrt{p}\left(\frac{k}{p}\right)-1\right).\]
So (\ref{sum-Rp}) holds.

\par As in \cite[Section 16.3]{Ireland-Rosen1990}, let \(m\) be a fixed positive integer, given a homomorphism
$\chi':\ (\mathbb{Z}/m\mathbb{Z})^\times\rightarrow\mathbb{C}^*$, define
$\chi:\ \mathbb{Z}\rightarrow\mathbb{C}^*$ as follows: for each $n\in\mathbb{Z}$, set $\chi(n)=0$ if $\gcd(n,m)>1$,
and $\chi(n)=\chi'(n+m\mathbb{Z})$ if $\gcd(n,m)=1$,
such functions $\chi$ are called $Dirichlet\ characters\ modulo\ m$.
From the definition it can be checked that
\par(a) $\chi(n+m)=\chi(n)$ for all $n\in\mathbb{Z}$;
\par(b) $\chi(mn)=\chi(m)\chi(n)$ for all integers $m,n$;
\par(c) $\chi(n)\neq 0$ if and only if $\gcd(n,m)=1$.
\par\noindent The next lemma (see e.g. \cite[Proposition 16.3.2]{Ireland-Rosen1990})
reveals the orthogonality of Dirichlet characters,
which is essential for isolating arithmetic progressions.
\begin{lemma} \label{orthogonality}{\rm(orthogonality)} Let $\chi$ and $\psi$ be Dirichlet
characters modulo $m$, and $a,b\in\mathbb{Z}$. Then
\par\noindent{\rm(i)}$\sum\limits_{a=0}^{m-1}\chi(a)\overline{\psi(a)}=\phi(m)\delta(\chi,\psi)$;
\par\noindent{\rm(ii)}$\sum\limits_{\chi}\chi(a)\overline{\chi(b)}=\phi(m)\delta(a,b)$,
\par\noindent where $\phi(m)$ is the number of generator of $(\mathbb{Z}/m\mathbb{Z})^\times$, $\delta(\chi,\chi)=1$ and
$\delta(\chi,\psi)=0$ if $\chi\neq\psi$, $\delta(a,b)=1$ if $a\equiv b ({\rm mod}\ m)$ and
$\delta(a,b)=0$ if $a\not\equiv b ({\rm mod}\ m)$.
\end{lemma}

\subsection{Linear congruence, distribution of primes and natural density}
The following two lemmas are crucial for proving Lemma \ref{lem:leg-same}. The first one is Chinese remainder theorem
(see e.g.\,\cite[Theorem 5.26]{Apostol1976}), which allows us to solve some systems of two or more linear congruences.
\begin{lemma}\label{lem:Chinese} {\rm (Chinese remainder theorem)} Assume $m_1,...,m_r$ are positive integers, relatively prime in pairs:
$(m_i, m_k)=1$ if $i\neq k$. Let $b_1,...,b_r$ be arbitrary integers. Then the system of congruences $x\equiv b_i({\rm mod}\ m_i)\ (i=1,...,r)$
has exactly one solution modulo the product $m_1\cdot\cdot\cdot m_r$.
\end{lemma}
The second one is Dirichlet's theorem collected from \cite[Chapter 4, Theorem 1.4 and 1.6]{Hindry2011},
which is fundamental for understanding the distribution of primes.
\begin{lemma}\label{lem:Dirichlet}{\rm (Dirichlet's theorem on arithmetic progressions)} Let $a,b\geq1$
be two relatively prime integers. Then there exist infinitely many primes $p$ satisfying $p\equiv b({\rm mod}\ a)$.
Moreover, as $x\rightarrow\infty$, we have the following asymptotic behavior
\begin{equation}\label{prime-theorem} \pi(x;a,b):={\rm card}\{p\ prime:\ p\leq x,\ p\equiv b({\rm mod}\ a)\}
\sim\frac{x}{\varphi(a)\log x},
\end{equation}
where $\varphi(b)$ is the number of generator of $(\mathbb{Z}/b\mathbb{Z})^\times$
{\rm(}called Euler totient of $b${\rm)}.
\end{lemma}

\par Similar to \cite[Remark 4.16]{Hindry2011}, let $\mathcal{B}$ be a set of certain primes
and $\mathcal{A} \subseteq \mathcal{B}$, define the \emph{natural density} of $\mathcal{A}$ in $\mathcal{B}$ as
\[ d_{\mathcal{B}}(\mathcal{A}):=\lim_{x\rightarrow\infty}\frac{{\rm card}\,\{p\in\mathcal{A}:\ \ p\le x\}}
{{\rm card}\,\{p\in\mathcal{B}:\ \ p\le x\}}.\]

\begin{lemma} \label{reciprocal-sum}Let $a,b\in\mathbb{N}^+$ with $\gcd(a,b)=1$ and
$\mathcal{B}=\{p\ {\rm prime}:\ p\equiv b({\rm mod}\,a)\}$. If $\mathcal{A}\subseteq\mathcal{B}$
with $d_{\mathcal{B}}(\mathcal{A})=\rho>0$, then the series $\sum_{p\in\mathcal{A}}\frac{1}{p}$ diverges. \end{lemma}

\begin{proof} Denote
\[\pi_{\mathcal{A}}(x)={\rm card}\,\{p\le x:\ p\in\mathcal{A}\}, \qquad\pi_{\mathcal{B}}(x)={\rm card}\,\{p\le x:\ p\in\mathcal{B}\}.\]
 Let $a_n=1$ if $n\in\mathcal{A}$, otherwise $a_n=0$.
Applying Abel's formula (see e.g.\,\cite[Chapter 4, Lemma 1.9]{Hindry2011}
to the sequence $a_n$ and the function $f(t)=1/t$ attains
\begin{equation}\label{Abel-A}\sum_{\substack{p\le x \\ p\in\mathcal{A}}} \frac{1}{p}
=\frac{\pi_{\mathcal{A}}(x)}{x}+\int_{1}^{x}\frac{\pi_{\mathcal{A}}(t)}{t^2}\,dt.\end{equation}

\par Since $d_{\mathcal{B}}(\mathcal{A})=\rho>0$ implies there exists a constant $X_0>1$
such that $ \pi_{\mathcal{A}}(x)\ge\frac{\rho}{2}\pi_{\mathcal{B}}(x)$ for all $x\ge X_0$.
Combining this with (\ref{Abel-A})produces
\begin{equation}\label{A>}\sum_{\substack{p\le x \\ p\in\mathcal{A}}} \frac{1}{p}
\ge\frac{\frac{\rho}{2}\pi_{\mathcal{B}}(x)}{x}+\int_{X_0}^{x}\frac{{\frac{\rho}{2}\pi_\mathcal{B}}(t)}{t^2}\,dt
+\int_{1}^{X_0}\frac{\pi_{\mathcal{A}}(t)}{t^2}\,dt.\end{equation}
Also, as $t\rightarrow\infty$ Lemma \ref{lem:Dirichlet} suggests \[\frac{\pi_{\mathcal{B}}(t)}{t^2}\sim\frac{1}{\varphi(a)t\log t}.\]
Hence, letting $x\rightarrow\infty$ for both sides of (\ref{A>}) yields $\sum_{p\in\mathcal{A}} 1/p$ diverges.
\end{proof}

\subsection{Oscillatory integral}
The lemma below from \cite[Proposition 3, Chapter VIII ]{Stein1993a} provides a rigorous asymptotic expansion for oscillatory integrals near a degenerate stationary point,
which is essential for analyzing the decay and singularities of Fourier transforms.
\begin{lemma}\label{Stein}
Let $k\geq 2$ and suppose $\phi(x_0)=\cdots=\phi^{(k-1)}(x_0)=0$ while $\phi^{(k)}(x_0)\neq 0$.
If the support of $\psi$ is contained in a sufficiently small neighborhood of $x_0$,
then as $\lambda\to\infty$ we have
\[I(\lambda):=\int e^{i\lambda\phi(x)}\psi(x)\,dx \;\sim\;
  \lambda^{-1/k}\sum_{j=0}^{\infty}a_j\lambda^{-j/k},\]
where the coefficients $a_j$ are uniquely determined by $\phi$ and $\psi$.
More precisely, for all non-negative integers $N$ and $r$ we have
\[\left(\frac{d}{d\lambda}\right)^{\!r}\!\left[I(\lambda)-\lambda^{-1/k}\sum_{j=0}^{N}a_j\lambda^{-j/k}\right]
=O\!\left(\lambda^{-r-(N+1)/k}\right)\]
as $\lambda\to\infty$. In particular, it follows for $r=0$, $k=2$ and $N=1$ that
\[I(\lambda)=\left(\frac{2\pi}{-i\phi''(x_0)}\right)^{\!1/2}\psi(x_0)\,\lambda^{-1/2}+O(\lambda^{-1}).\] \end{lemma}

\subsection{M\"obius function }
As in \cite[Chapter 2]{Apostol1976}, the M\"obius function $\mu$ is defined for positive
integers $n$ as follows
\[\mu(n) =
\begin{cases}
1 & \text{if } n = 1,\\
(-1)^k & \text{if } n = p_1 p_2 \cdots p_k \text{ with distinct primes } p_i,\\
0 & \text{otherwise}.
\end{cases}
\]
\par The following M\"obius inversion formula (see e.g.\,\cite[Theorem 2.9]{Apostol1976})
allows one to uniquely recover an arithmetic function from its divisor-sum transform.
\begin{lemma}{\rm (M\"obius inversion formula )} \label{Mobius}
Let $f$ and $g$ be two arithmetic functions. Then
\[g(n)=\sum_{d \mid n} f(d) \quad \Longleftrightarrow \quad
f(n)=\sum_{d \mid n} \mu(d) g\!\left(\frac{n}{d}\right).\]
\end{lemma}

\section{Lower bound for the lower box dimension}\label{sec:lower}
In this section, at the beginning we estimate the error of an infinite series and oscillatory integral, and
then prove some key lemmas, which will be adopted to establish lower bound for the lower box dimension.

\subsection{Auxiliary estimates}
By using the Poisson summation formula and detailed stationary phase analysis of oscillatory integrals,
we first prove the following lemma.
\begin{lemma}\label{lem:error estimate}
 Given a prime $p$, $k\in\mathbb{Z}\setminus\{0\}$ and $0<\delta\le 1$. Define\[S_k(p)=\sum_{n=1}^\infty\frac{1-e^{2\pi ik
n^2/p^2}}{n^{1+\delta}},\qquad M_k=\int_0^\infty\frac{1-e^{2\pi i k u^2}}{u^{1+\delta}}\,du,\]and set $E_k(p)=p^\delta S_k(p)-M_k$.
Then there exists a constant $C_\delta>0$
depending only on $\delta$ such that for any $0<\alpha\leq1$ we have
$$ |E_k(p)|\leq C_\delta\,|k|^{(\alpha+\delta)/2}\,p^{-\alpha}.$$
\end{lemma}
\begin{proof}
\textbf{Step 1.} Firstly, we prove there exists a constant $\widetilde{C}_{\delta}>0$
depending only on $\delta$ such that
\[
  |E_k(p)|\leq \widetilde{C}_{\delta}\,|k|^{(1+\delta)/2}\,p^{-1}.
\]
Since $E_{-k}(p)=\overline{E_k(p)}$, we may assume $k>0$.

\textbf{Case 1: $0<\delta<1$.} For given $k\in\mathbb{Z}\setminus\{0\}$, consider the function
\begin{equation}\label{hku} h_k(u)=
  \begin{cases}
    \dfrac{1-e^{2\pi iku^2}}{|u|^{1+\delta}}, & u\neq 0,\\[4pt]
    0, & u=0.
  \end{cases}\end{equation}
Set $M_k=\int_0^{\infty}h_k(u)\,du$.
Since $h_k$ is even, applying Poisson summation formula (see e.g.\,\cite[Theorem 3.2.8]{Grafakos2014}) gives
\[
  \frac{1}{p}\sum_{n\in\mathbb{Z}}h_k\left(\frac{n}{p}\right)
  =\sum_{m\in\mathbb{Z}}\widehat{h}_k(mp),
  \qquad
  \widehat{h}_k(\xi)=2\int_0^{\infty}h_k(u)\cos(2\pi\xi u)\,du,
\]
so that
\[E_k(p)=\frac{1}{p}\sum_{n=1}^{\infty}h_k\left(\frac{n}{p}\right)-M_k=\sum_{m=1}^{\infty}\widehat{h}_k(mp).
\]
Performing change of variable $v=\sqrt{k}\,u$ and setting $\eta=\frac{\xi}{\sqrt{k}}$ yield
\[\widehat{h}_k(\xi)=2k^{\delta/2}\Phi(\eta),\]
where
\[\Phi(\eta)=\int_0^{\infty}\frac{1-e^{2\pi iv^2}}{v^{1+\delta}}\cos(2\pi\eta v)\,dv.\]
Therefore,
\begin{equation} \label{Ekp} E_k(p)=2k^{\delta/2}\sum_{m=1}^{\infty}\Phi\!\left(\frac{mp}{\sqrt{k}}\right).
\end{equation}
\par Next, we shall show that there exist two constants $A_{\delta}>0$ and $\beta>1$
depending only on $\delta$ such that
\[|\Phi(\eta)| \leq A_{\delta} \eta^{-\beta}, \quad \forall \eta \geq 1.\]
Let $\varphi(v)=(1-e^{2\pi iv^2})v^{-1-\delta}$ for $v>0$.
We may assume $\eta>4$ and decompose $\Phi(\eta)$ as
\begin{eqnarray}\label{Phieta}\Phi(\eta)&=&\int_0^1\varphi(v)\cos(2\pi\eta v)\,dv+
\int_1^{\eta/4}\varphi(v)\cos(2\pi\eta v)\,dv+
\int_{\eta/4}^{\infty}\varphi(v)\cos(2\pi\eta v)\,dv\nonumber \\
&:=&I_1(\eta)+I_2(\eta)+I_3(\eta).\end{eqnarray}

\par \underline{\textit{Estimation of $I_1(\eta)$.}}\quad Integration by parts derives
\[I_1(\eta)=-\frac{1}{2\pi\eta}\int_0^1\varphi'(v)\sin(2\pi\eta v)\,dv.\]
By expanding $e^{2\pi iv^2}$ at the origin, one finds for all $v\in[0,1]$ that
\[\varphi'(v)=-2\pi i(1-\delta)v^{-\delta}+2\pi^2(3-\delta)v^{2-\delta}+R(v),\]
where $|R(v)|\leq Cv^{4-\delta}$ and $|R'(v)|\leq C'v^{3-\delta}$ with two constants $C,\ C'>0$
depending only on $\delta$.
\par For $\alpha>-2$ set $J_\alpha(\eta)=\int_0^1 v^\alpha\sin(2\pi\eta v)\,dv$. Then
\[I_1(\eta)
  =\frac{1}{\eta}\Bigl[i(1-\delta)J_{-\delta}(\eta)-\pi(3-\delta)J_{2-\delta}(\eta)
    -\frac{1}{2\pi}\int_0^1 R(v)\sin(2\pi\eta v)\,dv
  \Bigr].\]
Making the substitution $w=\eta v$ gives
\begin{align*}
J_{-\delta}(\eta)
&=\eta^{\delta-1}\int_{0}^{\eta}w^{-\delta}\sin(2\pi w)\,dw\\
&=\eta^{\delta-1}\Bigl(\int_{0}^{\infty}w^{-\delta}\sin(2\pi w)\,dw
  -\int_{\eta}^{\infty}w^{-\delta}\sin(2\pi w)\,dw\Bigr).
\end{align*}
Integration by parts shows
\[
\Bigl|\int_{\eta}^{\infty}w^{-\delta}\sin(2\pi w)\,dw\Bigr|
\leq\frac{1}{\pi}\eta^{-\delta}.\]
Hence,
\[J_{-\delta}(\eta)=\widetilde{C}_\delta\eta^{\delta-1}+O(\eta^{-1})\]
with constant \[\widetilde{C}_\delta=\int_{0}^{\infty}w^{-\delta}\sin(2\pi w)\,dw
=(2\pi)^{\delta-1}\Gamma(1-\delta)\sin(\pi\delta/2).\]

\par Also, applying integration by parts gets
\[J_{2-\delta}(\eta)=O(\eta^{-1}), \qquad \int_0^1 R(v)\sin(2\pi\eta v)\,dv=O(\eta^{-1}).\]
This implies
\begin{equation}\label{I1eta} I_1(\eta)=C_\delta^{\prime}\,\eta^{\delta-2}+O(\eta^{-2})\end{equation}
with constant
\[C_\delta^{\prime}=i(1-\delta)\widetilde{C}_\delta
=i(1-\delta)(2\pi)^{\delta-1}\Gamma(1-\delta)\sin(\pi\delta/2).\]

\par \underline{\textit{Estimation of $I_2(\eta)$.}}\quad Note that \(
\varphi^{\prime}(v)=-(1+\delta)v^{-2-\delta}(1-e^{2\pi i v^2})-4\pi i v^{-\delta}e^{2\pi i v^2}
\) and the phases $v^2\pm\eta v$ have no stationary points on $[1,\eta/4]$.
Integration by parts on this interval, together with $\sin(2\pi\eta v)=\frac{1}{2i}\big(e^{2\pi i\eta v}-
e^{-2\pi i\eta v}\big)$, yields
\begin{eqnarray} \label{I2eta} I_2(\eta)&=&\frac{\varphi(v)\sin(2\pi\eta v)}{2\pi\eta}\Big|_1^{\eta/4}
-\frac{1}{2\pi\eta}\int_1^{\eta/4}\varphi^{\prime}(v)\sin(2\pi\eta v)\ dv\nonumber \\
&=&\frac{\varphi(\eta/4)\sin(\pi\eta^2/2)}{2\pi\eta}+\frac{1+\delta}{2\pi\eta}
\int_1^{\eta/4}v^{-2-\delta}\sin(2\pi\eta v)\,dv\nonumber \\
&&\ -\frac{1+\delta}{4\pi i\eta}\int_1^{\eta/4}v^{-2-\delta}
\big(e^{2\pi i(v^2+\eta v)}-e^{2\pi i(v^2-2\eta v)}\big)\,dv\nonumber \\
&&\ +\frac{1}{\eta}\int_1^{\eta/4}v^{-\delta}
\big(e^{2\pi i(v^2+\eta v)}-e^{2\pi i(v^2-2\eta v)}\big)\,dv\nonumber \\
&=&O(\eta^{-2-\delta})+\frac{1+\delta}{2\pi\eta}\cdot O(\eta^{-1})
-\frac{1+\delta}{4\pi i\eta}\cdot O(\eta^{-1})+\frac{1}{\eta}\cdot O(\eta^{-1})\nonumber \\
&=&O(\eta^{-2}).\end{eqnarray}

\par\underline{\textit{Estimation of $I_3(\eta)$.}}\quad
Note that $\cos(2\pi\eta v)=\frac{1}{2}\big(e^{2\pi i\eta v}+
e^{-2\pi i\eta v}\big)$ and write $I_3(\eta)$ as
\[I_3(\eta)=\int_{\eta/4}^{\infty}v^{-1-\delta}\cos(2\pi\eta v)\,dv
-\frac{1}{2}\int_{\eta/4}^{\infty}v^{-1-\delta}e^{2\pi i(v^2+\eta v)}\,dv
-\frac{1}{2}\emph{}\int_{\eta/4}^{\infty}v^{-1-\delta}
e^{2\pi i(v^2-\eta v)}\,dv.\]
Again, using integration by parts obtains
\[\int_{\eta/4}^{\infty}v^{-1-\delta}\cos(2\pi\eta v)\,dv=O(\eta^{-2-\delta}),
\quad \int_{\eta/4}^{\infty}v^{-1-\delta}e^{2\pi i(v^2+\eta v)}\,dv=O(\eta^{-2-\delta}).\]
Denote $K(\eta)=\int_{\eta/4}^{\infty}v^{-1-\delta}
e^{2\pi i(v^2-\eta v)}\,dv$ and substitute $v=\eta x$ to get
\[K(\eta)=e^{-\pi i\eta^2/2}\eta^{-\delta}\int_{1/4}^{\infty}
x^{-1-\delta}e^{2\pi i\eta^2(x-\frac12)^2}\,dx.\]

Choose a function $\chi(x)\in C_0^\infty$ such that $\operatorname{supp}\chi\subseteq[\frac{5}{16},\frac{11}{16}]$
and $\chi(x)\equiv1$ for $x\in[\frac{7}{16},\frac{9}{16}]$. Then
\begin{align*}
&\int_{\frac{1}{4}}^{\infty}x^{-1-\delta}e^{2\pi i\eta^2\left(x-\frac12\right)^2}\,dx\\
=\,&\int_{\frac{1}{4}}^{\infty}x^{-1-\delta}\chi(x)e^{2\pi i\eta^2\left(x-\frac12\right)^2}\,dx
+\int_{\frac{1}{4}}^{\infty}x^{-1-\delta}(1-\chi(x))e^{2\pi i\eta^2\left(x-\frac12\right)^2}\,dx\\
:=\,&K_1(\eta)+K_2(\eta).
\end{align*}
\par For $K_1(\eta)$, applying Lemma \ref{Stein} with $\phi(x)=(x-\frac12)^2$ and
$\psi(x)=x^{-1-\delta}\chi(x)$ gives
\[K_1(\eta)=2^{\delta+\frac{1}{2}}\,e^{i\frac{\pi}{4}}\eta^{-1}+O(\eta^{-2}).\]
Since $1-\chi(x)\equiv 0$ for $x\in[\frac{7}{16},\frac{9}{16}]$, we have
\begin{align*}
K_2(\eta)
&=\int_{\frac{1}{4}}^{\frac{7}{16}}x^{-1-\delta}(1-\chi(x))e^{2\pi i\eta^2\left(x-\frac{1}{2}\right)^2}\,dx
+\int_{\frac{9}{16}}^{\infty}x^{-1-\delta}(1-\chi(x))e^{2\pi i\eta^2\left(x-\frac{1}{2}\right)^2}\,dx\\
&:=K_{2,1}(\eta)+K_{2,2}(\eta).
\end{align*}
\par Also, notice that the phase $(x-\frac{1}{2})^2$ has no stationary point on $[\frac{1}{4}, \frac{7}{16}]
\cup[\frac{7}{16},\infty)$,
carrying integration by parts shows $K_{2,1}(\eta)=O(\eta^{-2})$ and $K_{2,2}(\eta)=O(\eta^{-2})$, then
$K_2(\eta)=O(\eta^{-2})$. Collecting the estimates above we obtain
\begin{equation}\label{I3eta} I_3(\eta)=-2^{\delta-\frac{1}{2}}e^{\pi i(\frac{1}{4}-\frac{\eta^2}{2})}\eta^{-1-\delta}
+O(\eta^{-2-\delta}).\end{equation}
\par Consequently,  it follows from (\ref{Phieta})-(\ref{I3eta}) that
\[\Phi(\eta)=C_\delta^{\prime}\,\eta^{\delta-2}-2^{\delta-\frac{1}{2}}e^{\pi i(\frac{1}{4}-\frac{\eta^2}{2})}\eta^{-1-\delta}
+O(\eta^{-2})\]
and thus
\[|\Phi(\eta)|=O(\eta^{-\beta})\quad {\rm with}\quad \beta:=\min(1+\delta,2-\delta)>1.\]
Therefore, there exists a constant $A_{\delta}>0$ depending only on $\delta$ such that
\[|\Phi(\eta)| \leq A_{\delta} \eta^{-\beta}\quad {\rm for}\quad \forall\ \eta \geq 1.\]

\par Also, observe that $|\Phi(\eta)|\leq C_\delta^{''}$ for all $\eta\in\mathbb{R}$
with some constant $C_\delta^{''}>0$ depending only on $\delta$.
Splitting the sum at $m=\lfloor\sqrt{k}/p\rfloor$ leads to
\begin{align*}
  \sum_{m=1}^{\infty}\left|\Phi\!\left(\frac{mp}{\sqrt{k}}\right)\right|
  &= \sum_{m \leq \sqrt{k}/p} \left| \Phi\left(\frac{mp}{\sqrt{k}}\right) \right| +
  \sum_{m > \sqrt{k}/p} \left| \Phi\left(\frac{mp}{\sqrt{k}}\right) \right|\\
  &\leq C_\delta^{''}\,\frac{\sqrt{k}}{p}
     +A_\delta\!\left(\frac{p}{\sqrt{k}}\right)^{-\beta}\int_{\sqrt{k}/p}^{+\infty}t^{-\beta}\,dt\\
  &\leq\left(C_\delta^{''}+\frac{A_\delta}{\beta-1}\right)\frac{\sqrt{k}}{p}.
\end{align*}
This, together with (\ref{Ekp}), proves
\[|E_k(p)|\leq \widetilde{C}_{\delta}\,k^{(1+\delta)/2}\,p^{-1} \quad {\rm with}\quad \widetilde{C}_{\delta}
=2\left(C_\delta^{''}+\frac{A_\delta}{\beta-1}\right).\]

\par \textbf{Case 2: $\delta=1$.} Denote
\begin{equation}\label{hu} H_k(u)=\begin{cases}\dfrac{1-e^{2\pi iku^2}}{u^2}, & u\neq 0,\\[4pt]-2\pi ik, & u=0.
\end{cases}\end{equation}
Since $H_k$ is even again, adopting Poisson summation formula results in
\[E_k(p)=\sum_{m=1}^{\infty}\widehat{H}_k(mp)+\frac{\pi ik}{p},\]
where $\widehat{H}_k(\xi)=2\sqrt{k}\,\Phi(\eta)$ with $\eta=\xi/\sqrt{k}$ and
\[\Phi(\eta)=\int_0^{\infty}\frac{1-e^{2\pi iv^2}}{v^2}\cos(2\pi\eta v)\,dv.\]
\par As in Case 1, decompose $\Phi(\eta)=I_1(\eta)+I_2(\eta)+I_3(\eta)$. In this case,
note that for $I_1(\eta)$ the Taylor expansion of $\varphi'$ now becomes
\[\varphi'(v) = 4\pi^2 v + R(v)\quad {\rm for\ all}\quad v\in [0,1],\]
where $|R(v)|\leq C v^3$ and $|R'(v)|\leq C' v^2$ with two constants $C,\ C'>0$.
Repeating the same integration-by-parts arguments in Case 1 yields
\[I_1(\eta)=O(\eta^{-2})\quad{\rm and}\quad I_2(\eta)=O(\eta^{-2}).\]
\par Applying Lemma \ref{Stein} again for $I_3(\eta)$ gets
\[I_3(\eta)=-\sqrt{2}\,e^{\pi i(\frac{1}{4}-\frac{\eta^2}{2})}\eta^{-2}+O(\eta^{-3}).\]
Hence, $\Phi(\eta)=O(\eta^{-2})$. Utilizing the same splitting argument obtains
\[\sum_{m=1}^{\infty}\left|\Phi\!\left(\frac{mp}{\sqrt{k}}\right)\right|
  \leq C^{''}\,\frac{\sqrt{k}}{p},\]
and therefore \[|E_k(p)|\leq C\,|k|^{(1+\delta)/2}\,p^{-1}\emph{}
\quad {\rm with}\quad C=2C^{''}.\]

\par \textbf{Step 2.} Secondly, operating the change of variable $u=x/\sqrt{k}$ acquires
\[M_{k}=\int_{0}^{\infty}\frac{1-e^{2\pi iku^{2}}}{u^{1+\delta}}du=k^{\frac{\delta}{2}}\int_{0}^{\infty}\frac{1-e^{2\pi i x^{2}}}{x^{1+\delta}}dx=k^{\frac{\delta}{2}}M_{1}.\]
Additionally, by setting $h=\sqrt{k}/p$ we derive
\[
p^{\delta}S_{k}(p)=\frac{1}{p}\sum_{n=1}^{\infty}\frac{1-e^{2\pi i k\left(\frac{n}{p}\right)^{2}}}{\left(\frac{n}{p}\right)^{1+\delta}}=k^{\frac{\delta}{2}}h\sum_{n=1}^{\infty}\frac{1-e^{2\pi i(nh)^{2}}}{(nh)^{1+\delta}}.
\]
Therefore,
\[
E_{k}(p)=p^{\delta}S_{k}(p)-M_{k}=k^{\frac{\delta}{2}}\left(h\sum_{n=1}^{\infty}\varphi(nh)-\int_{0}^{\infty}\varphi(x)dx\right)
=:k^{\frac{\delta}{2}}R(h),\]
where $\varphi(v)=(1-e^{2\pi iv^2})v^{-1-\delta}$ is the same as in the previous step.

We next show
\[
|E_{k}(p)|\leq C_{\delta}k^{\frac{\alpha+\delta}{2}}p^{-\alpha},\quad \forall\,0<\alpha\leq1.
\]
Since $E_{k}(p)=k^{\frac{\delta}{2}}R(h)$, it suffices to verify $|R(h)|\leq C_{\delta}h^{\alpha}$ for all $h>0$.

\textbf{Case 1.} $0<h\leq1$.

From the preceding bound $|E_{k}(p)|\lesssim \widetilde{C}_{\delta}k^{\frac{1+\delta}{2}}p^{-1}$,
we have $|R(h)|\leq\widetilde{C}_{\delta}h$ for all $h>0$.
Thus for $h\in(0,1]$, $|R(h)|\leq\widetilde{C}_{\delta}h\leq\widetilde{C}_{\delta}h^{\alpha}$.

\textbf{Case 2.} $h>1$.

In this case
\[
\begin{aligned}
|R(h)|&=\left|h\sum_{n=1}^{\infty}\frac{1-e^{2\pi i(nh)^{2}}}{(nh)^{1+\delta}}-\int_{0}^{\infty}
\frac{1-e^{2\pi i x^{2}}}{x^{1+\delta}}dx\right|\\
&=\left|h^{-\delta}\sum_{n=1}^{\infty}\frac{1-e^{2\pi i n^{2}h^{2}}}{n^{1+\delta}}-M_{1}\right|\\
&\leq2h^{-\delta}\zeta(1+\delta)+|M_{1}|\\
&\leq2\zeta(1+\delta)+|M_{1}|.
\end{aligned}
\]
Combining both cases attains $|R(h)|\leq C_{\delta}h^{\alpha}$ for all $h>0$, where
\[C_{\delta}=\max\left\{\widetilde{C}_{\delta},2\zeta(1+\delta)+|M_{1}|\right\}.
\]
\end{proof}

\subsection{Key lemmas}
The next two lemmas are crucial for proving the lower box dimension, which translate the nonvanishing of a certain arithmetic
series into a uniform lower bound
for weighted Legendre sums.  They are purely number-theoretic. We first prove a finite version (Lemma~\ref{lem:leg-same}) and then extend it to
infinite sums (Lemma~\ref{lem:legendre-sum}).

\begin{lemma}\label{lem:leg-same}
Let $\{r_k\}_{|k|\le N}$ be complex numbers with $r_0=0$ and $r_k=\overline{r_{-k}}$.
Let $D^{+}$ be the set of square-free integers in $[1,N]$.
If $R_{d_0,N}:=\sum\limits_{1\le m^{2}d_0\le N}r_{m^{2}d_0}\neq0$ for some $d_0\in D^{+}$,
then there exist positive integers $a,b$ with $0<b<a$ and a set of odd primes
\[\mathcal{P}=\{\text{odd\ prime}\ p:\ p\equiv b\,({\rm mod}{a})\}\]
such that
\par {\rm (i)} $\left(\frac{k}{p_1}\right)=\left(\frac{k}{p_2}\right)$ for each given $k$ with $|k|\le N$ and all $p_1,\,p_2\in\mathcal{P}$;
\par {\rm (ii)} $\Big|\sum\limits_{1\le|k|\le N}r_k\left(\frac{k}{p}\right)\Big|\ge \sqrt{2}\,|R_{d_0,N}|$ for every $p\in\mathcal{P}$.
\end{lemma}

\begin{proof}
\textbf{Step 1: Proof of (i).} Let $k$ be an integer whose prime factors do not exceed $N$, then $k$ can be factorized as
\[k=\pm2^{s_0}\prod_{i=1}^rq_i^{s_i}\]
with primes $q_i\le N$. Applying the complete multiplicativity of the Legendre symbol for each prime $p$ implies
\[\left(\frac{k}{p}\right)=\left(\frac{-1}{p}\right)^l\left(\frac{2}{p}\right)^{s_0}\prod_{i=1}^r\left(\frac{q_i}{p}\right)^{s_i},\]
where $l=1$ if $k<0$; otherwise, $l=0$. By Lemma \ref{lem:leg-law}(i)-(ii), we have
\[\left(\frac{-1}{p}\right)=(-1)^{\frac{p-1}{2}},\quad \left(\frac{2}{p}\right)=(-1)^{\frac{p^2-1}{8}}\]
and \[\left(\frac{q_i}{p}\right)=\left(\frac{p}{q_i}\right)(-1)^{\frac{(p-1)(q_i-1)}{4}}.\]
Hence, the value of $\left(\frac{k}{p}\right)$
is completely determined by the residue of $p$
modulo $8$ and modulo every odd prime $q\le N$.
\par Setting
\begin{equation} \label{a} a:=8\cdot\prod_{\substack{q\le N \\ q\ \text{odd\ prime}}}q,\end{equation}
it follows that $\left(\frac{k}{p}\right)$ depends only on the residue class of $p$ modulo $a$. In particular,
(i) holds for $\mathcal{P}$ taken to be any set of primes lying in a single residue class modulo $a$,
regardless of which residue class is chosen. The specific residue class will be determined in Step 3.

\par \textbf{Step 2: Reformulation in terms of Dirichlet characters.}
Define \[D^-:=-D^+=\{-d:\ d\in D^+\}.\]
Since $\mathcal{P}$ contains infinitely many primes by Dirichlet's theorem (Lemma \ref{lem:Dirichlet}),
it suffices to consider $p \in \mathcal{P}$ with $p > N$. For such $p$ and
any $k = m^2 d$ with $1\le m^2 d \le N$, we have $m \le \sqrt{N} < p$, hence
$\gcd(m, p) = 1$, and the complete multiplicativity of the Legendre symbol gives
\[\left(\frac{k}{p}\right)=\left(\frac{m^2d}{p}\right)= \left(\frac{m}{p}\right)^2\left(\frac{d}{p}\right)=\left(\frac{d}{p}\right).\]

\par Hence, the sum $S(p):=\sum\limits_{1\le|k|\le N}r_k\left(\frac{k}{p}\right)$ satisfies
\[S(p)=\sum_{d\in D^{+}}R_{d}\left(\frac{d}{p}\right)+\sum_{d\in D^{-}}R_{d}\left(\frac{d}{p}\right),\]
where $R_{d}=\sum\limits_{1\leq m^{2}d\leq N}r_{m^{2}d}$ if $d\in D^{+}$, and
$R_{d}=\sum\limits_{-N\leq m^{2}d\leq-1}r_{m^{2}d}$ if $d\in D^{-}$.

\par Note that $r_k=\overline{r_{-k}}$. Writing $d=-d'$ for $d\in D^{-}$ with $d'\in D^{+}$
and setting $\eta_p=\left(\frac{-1}{p}\right)\in\{\pm1\}$ yield
\[S(p)=\sum_{d\in D^{+}}\bigl(R_{d}+\eta_p\,\overline{R_{d}}\bigr)\left(\frac{d}{p}\right).\]

\par Now consider the finite group \(G = (\mathbb{Z}/a\mathbb{Z})^{\times}\), where $a$ is the same as in (\ref{a}).
Given $d\in D^{+}$, define
$\chi_{d}\colon G\to\{\pm1\}$ by
\[\chi_{d}(b)=\left(\frac{d}{p}\right)\quad {\rm for\ each}\ b\in G\ {\rm and\ any\ prime}\ p\equiv b\,({\rm mod}\ a).\]
By Step 1, the function $\chi_{d}$ is well defined on $G$.

\par Firstly, we affirm that $\chi_d$ is a Dirichlet character on $G$.
To verify multiplicativity, fix $b_1,b_2\in G$
and take primes $p_j\equiv b_j\pmod{a}$ for $j=1,2$ and $p_{12}\equiv b_1b_2\pmod{a}$.
Suppose $d$ has prime factorization $d=q_1\cdots q_r$, it suffices to show
$\left(\frac{q_i}{p_{12}}\right)=\left(\frac{q_i}{p_1}\right)\left(\frac{q_i}{p_2}\right)$ for each $1\leq i\leq r$.
Since $8q_i\mid a$, we have $p_{12}\equiv p_1p_2\pmod{8q_i}$.
\par If $q_i$ is an odd prime, then applying Lemma \ref{lem:leg-law}(i)-(ii) with
$\frac{p_1p_2-1}{2}\equiv\Big(\frac{p_1-1}{2}+\frac{p_2-1}{2}\Big)\pmod{2}$ and
complete multiplicativity of the Legendre symbol gains
\begin{eqnarray} \label{qip12} \left(\frac{q_i}{p_{12}}\right)&=&\left(\frac{p_{12}}{q_i}\right)(-1)^{\frac{(p_{12}-1)(q_i-1)}{4}}\nonumber \\
&=&\left(\frac{p_1}{q_i}\right)\left(\frac{p_2}{q_i}\right)(-1)^{\frac{(p_1-1)(q_i-1)}{4}}(-1)^{\frac{(p_2-1)(q_i-1)}{4}}\nonumber\\
&=&\left(\frac{q_i}{p_1}\right)\left(\frac{q_i}{p_2}\right).
\end{eqnarray}

\par If $q_i=2$, using Lemma \ref{lem:leg-law}(i) and $\frac{(p_1p_2)^2-1}{8}\equiv
\Big(\frac{p_1^2-1}{8}+\frac{p_2^2-1}{8}\Big)\pmod{2}$
guarantees (\ref{qip12}) also holds. Hence,
\[\chi_d(b_1b_2)=\left(\frac{d}{p_{12}}\right)=\prod_{i=1}^{r}\left(\frac{q_i}{p_{12}}\right)\left(\frac{q_i}{p_{12}}\right)
=\prod_{i=1}^{r}\left(\frac{q_i}{p_1}\right)\left(\frac{q_i}{p_2}\right)=\chi_d(b_1)\chi_d(b_2),\]
which proves the multiplicativity of $\chi_d$.
Furthermore, since $p>N\geq d$ implies $\gcd(d,p)=1$, we have $\chi_d(b)\in\{\pm1\}$ for all
$b\in G$. Thus $\chi_d$ is a Dirichlet character on $G$.

\par Moreover, the characters $\chi_d$ are \emph{even}, i.e.\ $\chi_d(-b)=\chi_d(b)$ for all $b\in G$.
By multiplicativity of $\chi(d)$, it suffices to show $\chi_d(-1)=1$. Recall that $a$ is defined as in
(\ref{a}), then for any prime $p\equiv-1\pmod{a}$ we have $p\equiv -1\equiv q'-1\pmod{q'}$ for each odd prime $q'\le N$
and $\frac{p-1}{2}$ is odd. This, together with Lemma \ref{lem:leg-law}(i)-(ii), implies
\[\left(\frac{q'}{p}\right)=\left(\frac{p}{q'}\right)(-1)^{\frac{(p-1)(q'-1)}{4}}=\left(\frac{-1}{q'}\right)(-1)^{\frac{(p-1)(q'-1)}{4}}=
(-1)^{\frac{q'-1}{2}}\cdot(-1)^{\frac{q'-1}{2}}=1.\]
Hence, by the definition of $\chi_d$ we have
\[\chi_d(-1)=\prod_{q'\mid d}\left(\frac{q'}{p}\right)=1.\]

\par Finally, we prove that $\chi_{d_1} \neq \chi_{d_2}$ whenever $d_1 \neq d_2\in D^+$.
Since $d_1, d_2$ are distinct square-free positive integers, without loss of generality, we may assume there exists a prime $q$
 with $q \mid d_1$ and $q \nmid d_2$.

\textbf{Case 1}: If $q=2$, then $2\mid d_1$ and $2\nmid d_2$. By the Chinese Remainder Theorem (Lemma \ref{lem:Chinese}),
there exists a unique $b\in G$ such that
\[b\equiv 5({\rm mod}\,8)\quad\text{and}\quad b\equiv 1({\rm mod}\,q')\ \text{for each odd prime } q'\le N.\]
Fix any prime $p$ with $p\equiv b\pmod a$. Notice that (\ref{a}) indicates $8\mid a$ and $q'\mid a$ for every odd prime $q'\le N$,
this implies
\[p\equiv 5({\rm mod}\,8)\quad\text{and}\quad p\equiv 1({\rm mod}\,q')\ \text{for each odd prime } q'\le N.\]

\par Firstly, by $p\equiv5({\rm mod}\,8)$ and quadratic reciprocity law (Lemma \ref{lem:leg-law}(ii)) we have
\[\left(\frac{2}{p}\right)=(-1)^{\frac{p^2-1}{8}}=-1.\]
Next, fix an odd prime $q'\le N$. Note that $p\equiv5({\rm mod}\,8)$ suggests $\frac{p-1}{2}$ is even, and then
\[(-1)^{\frac{(p-1)(q'-1)}{4}}=\Big((-1)^{\frac{p-1}{2}}\Big)^{\frac{q'-1}{2}}=1,\]
regardless of the value of $q'$. Combining this with quadratic reciprocity law (Lemma \ref{lem:leg-law}(ii))
and $p\equiv b\equiv 1\pmod{q'}$ gives
\[
\left(\frac{q'}{p}\right)=\left(\frac{p}{q'}\right)(-1)^{\frac{(p-1)(q'-1)}{4}}=\left(\frac{p}{q'}\right)=\left(\frac{1}{q'}\right)=1.
\]

\par Therefore,
\[\chi_{d_1}(b)=\left(\frac{2}{p}\right)\prod_{\substack{q'\mid d_1\\ q'\neq 2}}\left(\frac{q'}{p}\right)=(-1)\cdot 1=-1,
\qquad \chi_{d_2}(b)=\prod_{q'\mid d_2}\left(\frac{q'}{p}\right)=1,\]
where the last product runs over odd primes only, since $2\nmid d_2$. Hence $\chi_{d_1}(b)\neq\chi_{d_2}(b)$.

\textbf{Case 2}: When $q$ is an odd prime with $q\mid d_1$ and $q\nmid d_2$, since $(\mathbb{Z}/q\mathbb{Z})^{\times}$ is cyclic of order $q-1\ge 2$, exactly half of its elements are quadratic non-residues (see e.\,g.\,\cite[Theorem 9.1]{Rosen1984}), we may fix one, say $c$, so that $\left(\frac{c}{q}\right)=-1$. Similar to Case 1, by the Chinese Remainder Theorem (Lemma \ref{lem:Chinese}), there exists a unique $b\in G$ such that $b\equiv1({\rm mod}\,8),\ b\equiv c({\rm mod}\,q)$ and $b\equiv 1({\rm mod}\,q^{\prime})$ for each odd prime $q'\le N$ with $q'\neq q$.
\par Also, given any prime $p\equiv b\pmod a$, then $p\equiv 1\pmod 8$ and $p\equiv b\pmod{q'}$ for every odd prime $q'\le N$.
As $p\equiv1\pmod8$ implies $(-1)^{\frac{(p-1)(q'-1)}{4}}=1$ for every odd prime $q'$, repeating the same argument in Case 1 derives
\[
\left(\frac{2}{p}\right)=1,\qquad
\left(\frac{q'}{p}\right)=\left(\frac{p}{q'}\right)=\left(\frac{b}{q'}\right)=
\begin{cases}
\left(\dfrac{c}{q}\right)=-1, & q'=q,\\[4pt]
1, & q'\neq q.
\end{cases}
\]
This, together with $q\mid d_1$ and $q\nmid d_2$ guarantees
\[
\chi_{d_1}(b)=\left(\frac{q}{p}\right)=-1,\qquad \chi_{d_2}(b)=1,
\]
so $\chi_{d_1}(b)\neq\chi_{d_2}(b)$.
This proves $\chi_{d_1} \neq \chi_{d_2}$ whenever $d_1 \neq d_2\in D^+$.

\textbf{Step 3: Proof of (ii).} Setting $S(b):=S(p)$ for $p\equiv b\pmod{a}$, we have
\[S(b)=\sum_{d\in D^{+}}\bigl(R_{d}+\eta_{b}\,\overline{R_{d}}\bigr)\chi_{d}(b),\]
where $\eta_b={\left(\frac{-1}{p}\right)}$ for any $p\equiv b\pmod{a}$.

\par Observe that for two cases in Step 2, both distinguishing elements $b$ belong to
the subgroup $G_1=\{b\in G : b\equiv1\pmod4\}$. Partition $G$ into the subgroup $G_1$ and
its coset \(G_3 = \{b\in G : b\equiv3\pmod4\}\).
For $b\in G_{1}$ we have $\eta_b=1$ and
\[S(b) = 2\sum_{d\in D^+}{\rm Re}(R_d)\,\chi_d(b);
\]for \(b\in G_3\) then \(\eta_b = -1\) and
\[S(b)=2i\sum_{d\in D^+} {\rm Im}(R_d)\,\chi_d(b).\]

\par Since $G_1$ is a subgroup of $G$
and each $\chi_d$ is a Dirichlet character on $G$, the restriction $\chi_d|_{G_1}$ is also
a Dirichlet character on $G_1$. Moreover, since the distinguishing element $b$ constructed
above lies in $G_1$ for every pair $d_1\neq d_2$, the restrictions $\chi_d|_{G_1}$ are
pairwise distinct. Applying the orthogonality relations (Lemma \ref{orthogonality}) on $G_1$
gets
\[\frac{1}{|G_1|}\sum_{b\in G_1}\chi_{d_1}(b)\,\overline{\chi_{d_2}(b)}
=\begin{cases}1,&d_1=d_2,\\0,&d_1\neq d_2,\end{cases}\]
and hence
\begin{equation} \label{sumG1} \sum_{b\in G_1}|S(b)|^{2}
=4\sum_{d_1,d_2\in D^{+}}{\rm Re}(R_{d_1}){\rm Re}(R_{d_2})\sum_{b\in G_1}\chi_{d_1}(b)\overline{\chi_{d_2}(b)}
=2\varphi(a)\sum_{d\in D^{+}}\big({\rm Re}(R_{d})\big)^2.\end{equation}
For $G_3$, though it is not a subgroup of $G$, the map $b\mapsto -b$ is a bijection from $G_1$ to $G_3$,
so using the evenness of each $\chi_d$ gains
\begin{equation} \label{sumG3}
\sum_{b\in G_3}|S(b)|^{2}
=4\sum_{d_1,d_2\in D^{+}}{\rm Im}(R_{d_1}){\rm Im}(R_{d_2})\sum_{b\in G_1}\chi_{d_1}(-b)\overline{\chi_{d_2}(-b)}
=2\varphi(a)\sum_{d\in D^{+}}\big({\rm Im}(R_{d})\big)^2.\end{equation}

\par Combining (\ref{sumG1}) and (\ref{sumG3}) yields
\[\sum_{b\in G}|S(b)|^{2}=2\varphi(a)\sum_{d\in D^{+}}|R_{d}|^{2}.\]
Thus, dividing by $|G|=\varphi(a)$obtains the average
\[\frac{1}{\varphi(a)}\sum_{b\in G}|S(b)|^{2}=2\sum_{d\in D^{+}}|R_{d}|^{2}\ge 2|R_{d_0,N}|^{2}.\]

Hence, there exists at least one residue \(b\in G\) with $|S(b)|\ge\sqrt{2}\,|R_{d_0,N}|$.
Choosing \(\mathcal{P}\) as the set of primes congruent to this \(b\) modulo \(a\) completes the proof of (ii).
\end{proof}

\begin{lemma}\label{lem:legendre-sum}
Let \(\{r_k\}_{k\in\mathbb{Z}}\) satisfy $r_0=0,\ r_k=\overline{r_{-k}}$ and $\sum_{k\in\mathbb{Z}}|r_k|<\infty$.
Let \(D^{+}\) be the set of all square-free positive integers.
If \(R_{d_0}:=\sum_{m\in\mathbb{N}^+}r_{m^{2}d_0}\neq0\) for some \(d_0\in D^{+}\), then there exist
\(a,b\in\mathbb{N}^+\) with $b<a$ and a set of primes
\[\mathcal{P}=\{\text{odd\ prime}\ p:\ p\equiv b\,({\rm mod}\,a)\}\]
such that for all \(p\in\mathcal{P}\) we have
\[\Bigl|\sum_{k\in\mathbb{Z}} r_k\left(\frac{k}{p}\right)\Bigr|\ge \eta>0,\]
where \(\eta=\bigl(\frac{\sqrt{2}}{2}-\frac12\bigr)|R_{d_0}|\).
\end{lemma}

\begin{proof} Since $\sum_{k\in\mathbb{Z}}|r_k|<\infty$, we can choose large $N\ge d_0$ such that
\(\sum_{|k|>N} |r_k| < \frac12
|R_{d_0}|\).
Define \( R_{d_0,N}:=\sum_{1\le m^{2}d_0\le N} r_{m^{2}d_0}\). Then by the choice of \(N\) and the triangle inequality we have
\[|R_{d_0,N}| \ge |R_{d_0}|-\sum_{|k|>N} |r_k| \geq \frac12 |R_{d_0}|.\]
Setting \(r_k=0\) for \(|k|>N\) and applying Lemma~\ref{lem:leg-same} to the finite sequence \(\{r_k\}_{|k|\le N}\),
we obtain integers \(a,b\) and a prime set \(\mathcal{P}\) such that for every \(p\in\mathcal{P}\)
\[\Bigl|\sum_{|k|\le N} r_k\left(\frac{k}{p}\right)\Bigr| \ge \sqrt{2}\,|R_{d_0,N}|.\]
For such \(p\) we have
\begin{align*}
\Bigl|\sum_{k\in\mathbb{Z}} r_k\left(\frac{k}{p}\right)\Bigr|
&\ge \Bigl|\sum_{|k|\le N} r_k\left(\frac{k}{p}\right)\Bigr| - \sum_{|k|>N} |r_k|\\
&\ge \sqrt{2}\,|R_{d_0,N}|-\tfrac12 |R_{d_0}|
\ge \bigl(\tfrac{\sqrt{2}}{2}-\tfrac12\bigr)|R_{d_0}|>0.
\end{align*}
Hence the desired lower bound holds with
\(\eta=\bigl(\frac{\sqrt{2}}{2}-\frac12\bigr)|R_{d_0}|>0.\)
\end{proof}

\par The following lemma will facilitate attaining the lower bound for
$\underline{\dim}_{\mathrm{B}}\bigl(graph(G_\delta)\bigr).$
\begin{lemma}\label{lem:fip} For sufficiently large integer $N$, take
\[\mathcal{P}_N=\bigl\{p\ {\rm prime}:\ C_1\sqrt{N}\leq p\leq C_2\sqrt{N} \bigr\}\]
with constants $0<C_1<C_2$. Define the intervals
\[D_{p,a}:=\Bigl[ \frac{a}{p},\, \frac{a}{p}+\frac{1}{p^2}\Bigr] \]
for each $p\in\mathcal{P}_N$ and $1\leq a\leq p-1$.
Let $F\colon [0,1]\to\mathbb{R}$ be continuous. Then
\[\underline{\dim}_{\mathrm{B}}\!\bigl(\operatorname{graph}(F)\bigr)
\;\geq\;
\varliminf_{N\to\infty}
\frac{\log\Bigl(N\displaystyle\sum_{p\in\mathcal{P}_N}\,\sum_{1\leq a\leq p-1}
\bigl|F\!\bigl(\tfrac{a}{p}\bigr)-F\!\bigl(\tfrac{a}{p}+\tfrac{1}{p^2}\bigr)\bigr|\Bigr)}{\log N}.
\]

\begin{proof}
The argument proceeds in the following three steps.

\par \textbf{Step 1: Intersection structure of $\mathcal{D}$.}
Set $\mathcal{D} = \{D_{p,a} : p\in\mathcal{P}_N,\, 1\leq a\leq p-1\}$.
We claim that each $D_{p_0,a_0}\in\mathcal{D}$ intersects at most $C$ other members of $\mathcal{D}$,
where $C$ is a constant depending only on $C_1$ and $C_2$.

\par Indeed, two intervals \(D_{p,a}\) and \(D_{p_0,a_0}\) intersect if and only if
\[\frac{a}{p} \leq \frac{a_0}{p_0}+\frac{1}{p_0^2} \quad \text{and} \quad \frac{a_0}{p_0} \leq \frac{a}{p}+\frac{1}{p^2},\]
which is equivalent to
\[\frac{a}{p}-\frac{a_0}{p_0}\in\left[-\frac{1}{p^2}, \frac{1}{p_{0}^2} \right],\]
that is,
\begin{equation}\label{intersection-criteria}-\frac{p_0}{p} \;\leq\; ap_0 - a_0p \;\leq\; \frac{p}{p_0}.\end{equation}
Set $j:=ap_0-a_0 p$. Since $p, p_0 \in [C_1\sqrt{N}, C_2\sqrt{N}]$, both ratios
$p_0/p$ and $p/p_0$ lie in $[C_1/C_2,\, C_2/C_1]$, thus $j$ must satisfy $|j| \leq C_2/C_1$.
Denote $K:=\lfloor C_2/C_1\rfloor + 1$, then $j$ ranges over at most $2K$ consecutive
integers.

\par For fixed $j,p_0$ and $a_0$, the equation $ap_0 - a_0 p = j$ gives
$a=(j+a_0 p)/p_0$, which requires $p_0\mid (j+a_0 p)$, i.e.,
\[a_0 p \;\equiv-j ({\rm mod}\ {p_0}).\]
Since $1\leq a_0\leq p_0-1$ and $p_0$ is prime, \(a_0^{-1} \pmod{p_0}\) exists, so the above congruence equation
is equivalent to $p\equiv -j a_0^{-1}\pmod{p_0}$. The number of primes \(p \in [C_1\sqrt{N}, C_2\sqrt{N}]\)
satisfying this congruence is at
most
\[
\left\lfloor\frac{(C_2-C_1)\sqrt{N}}{p_0}\right\rfloor + 1
\;\leq\; \left\lfloor\frac{C_2-C_1}{C_1}\right\rfloor + 1 \;=:\; M_0,
\]
which is an absolute constant. Summing over the at most $2K$ admissible values of $j$,
each $D_{p_0,a_0}$ intersects at most
\[C\,:=\; 2K M_0
\,=\;2\,\left(\left\lfloor\frac{C_2}{C_1}\right\rfloor+1\right)\,\left(\left\lfloor\frac{C_2-C_1}{C_1}\right\rfloor+1\right)
\]
other intervals in $\mathcal{D}$.

\par \textbf{Step 2: Bounding the number of $D_{p,a}$ meeting each partition interval.}
Let $I_k = [k/N, (k+1)/N]$ for $0\leq k\leq N-1$.
Next we shall show that for each $k$, at most $\widetilde{C}$ intervals in $\mathcal{D}$ meet $I_k$,
where $\widetilde{C}$ is a constant depending only on $C_1$ and $C_2$.

\par Actually, each $D_{p,a}\in\mathcal{D}$ has length $1/p^2\in[\delta,\Delta]$, where
\[\delta :=\frac{1}{C_2^2 N}, \qquad \Delta :=\frac{1}{C_1^2 N}.\]
If $D_{p,a}=[x, x+1/p^2]$ intersects $I_k=[k/N,\, (k+1)/N]$, then
\[
x \;\leq\; \frac{k+1}{N} \qquad\text{and}\qquad x+\frac{1}{p^2} \;\geq\; \frac{k}{N},
\]
so $x\in[k/N-\Delta,\, (k+1)/N]$, an interval of length $l+\Delta$ where $l=1/N$.

Partition the interval $[k/N-\Delta,\, (k+1)/N]$ into $R+1$ subintervals of length at most $\delta/2$,
where $R=\lfloor2(l+\Delta)/\delta\rfloor$ is a constant depending only on $C_1$ and $C_2$.
If two members $D_{p,a}$ and $D_{p',a'}$ of $\mathcal{D}$ have left endpoints $x, x'$ in the
same sub-interval, then $|x-x'|\leq \delta/2\leq 1/p^2=|D_{p,a}|$. Without loss of generality, assume
\(x<x'\). Then \(x'<x+\frac{\delta}{2}<x+1/p^2\), so \(x' \in D_{p,a}\), hence \(D_{p,a} \cap D_{p',a'} \neq \emptyset\).

\par We claim any subcollection $S\subset\mathcal{D}$ in which any two members intersect
satisfies $|S|\leq C+1$. Indeed, fix any given $D_0\in S$; every other element of $S$
intersects $D_0$, so by Step~1 there are at most $C$ such elements, giving $|S|\leq C+1$.
Therefore each sub-interval contains at most $C+1$ left endpoints, giving
\begin{equation}\label{eq:multiplicity}
{\rm card}\,\bigl\{(p,a):\ D_{p,a}\cap I_k\neq\emptyset\bigr\}
\;\leq\; (C+1)(R+1) \;=:\; \widetilde{C}.
\end{equation}

\par \textbf{Step 3: Lower bound on the box dimension.}
The minimal number $A_{N}(F)$ of squares in $\mathcal{M}_N$ required to cover \(\operatorname{graph}(F)\) satisfies
\[A_{N}(F)\;\geq\;N\sum_{k=0}^{N-1}\operatorname{osc}(F, I_k),\]
where $\operatorname{osc}(F, I)=\sup\limits_{x\in I} F(x)-\inf\limits_{x\in I} F(x)$.
Since $D_{p,a}\subset [a/p,\, a/p+1/p^2]$, for each pair $(p,a)$ we have
\[
\Bigl|F\!\Bigl(\frac{a}{p}\Bigr) - F\!\Bigl(\frac{a}{p}+\frac{1}{p^2}\Bigr)\Bigr|
\;\leq\;
\sum_{k\,:\,I_k\cap D_{p,a}\neq\emptyset}
\operatorname{osc}(F, I_k).
\]
Summing over all $(p,a)$ and exchanging the order of summation,
the multiplicity bound~\eqref{eq:multiplicity} gives
\[
\sum_{p\in\mathcal{P}_N}\,\sum_{1\leq a\leq p-1}
\Bigl|F\,\Bigl(\frac{a}{p}\Bigr) - F\!\Bigl(\frac{a}{p}+\frac{1}{p^2}\Bigr)\Bigr|
\;\leq\;
\widetilde{C}\sum_{k=0}^{N-1}\operatorname{osc}(F, I_k).
\]
\par Combining this with the box-counting lower bound yields
\[A_{N}(F)\;\geq\;\frac{N}{\widetilde{C}}\sum_{p\in\mathcal{P}_N}\,\sum_{1\leq a\leq p-1}
\Bigl|F\!\Bigl(\frac{a}{p}\Bigr) - F\!\Bigl(\frac{a}{p}+\frac{1}{p^2}\Bigr)\Bigr|
\]
Since $\widetilde{C}$ is a constant independent of $N$, we have
\begin{align*}
\underline{\dim}_{\mathrm{B}}(\operatorname{graph}(F))
&= \varliminf_{N \to \infty} \frac{\log A_N(F)}{\log N} \\
&\geq \varliminf_{N \to \infty} \frac{\log\left( N \sum\limits_{p\in \mathcal{P}_N}
\sum\limits_{1\leq a\leq p-1}\left| F\left(\frac{a}{p}\right)-F\left(\frac{a}{p}+\frac{1}{p^2}\right) \right| \right)}{\log N},
\end{align*}
as desired.
\end{proof}
\end{lemma}

\begin{remark}
{\rm It is worthwhile to point out intervals $D_{p,a}$ and $D_{p_0,a_0}$ can intersect for sufficiently large primes
$p$ and $p_0$.
Indeed, let $p_0, p$ be twin primes with $p=p_0+2$ and $p_0=2m+1$ for some integer $m\geq 1$.
Set $a_0=\lfloor p_0/2\rfloor=m$ and $a=\lfloor p/2\rfloor=m+1$ such that $ap_0-a_0 p=1$
and $1\in(-p_0/p,\, p/p_0)$, then (\ref{intersection-criteria}) indicates
$D_{p,a}\cap D_{p_0,a_0}\neq\emptyset$.}
\end{remark}

The following Lemma is crucial for estimating the error term in Theorem \ref{lower bound of dimension}.
\begin{lemma}\label{positive density} Let $\mathbb{P}$ be the set of all primes.
Suppose $\{a_k\}_{k\ge1}\subset\mathbb{C}$ satisfy $\sum_{k=1}^{\infty} |a_k| k^{\beta}<\infty$ for some constant $\beta>0$.
Then for any subset $\mathcal{P}\subset\mathbb{P}$ with natural density
$d_\mathbb{P}(\mathcal{P})>0$ and any constant $C>0$, the set
\[E_C:=\{ p\in\mathcal{P}:\ \sum_{p\mid k}|a_k|\le C\,p^{-\beta-1}\}\]
has natural density $d_{\mathcal{P}}(E_C)=1$.
\end{lemma}

\noindent\begin{proof}
Denote $M=\sum_{k=1}^{\infty} |a_k|k^{\beta}$ and
$S_p = \sum\limits_{\substack{k\ge 1\\ p\mid k}} |a_k|$ for each prime $p$.
Define the weighted sum
\begin{equation}\label{weighted sum}
W:=\sum_{p\in\mathcal{P}} p^{\beta}S_p
=\sum_{p\in\mathcal{P}} p^{\beta}\sum_{\substack{k\ge 1\\ p\mid k}} |a_k|
=\sum_{k=1}^{\infty}|a_k|\sum_{\substack{p\in\mathcal{P}\\ p\mid k}} p^{\beta}.
\end{equation}
Claim: there exists a constant $C_{\beta}>0$ depending only on $\beta$ such that for all $k\ge2$ we have
\begin{equation}\label{estimate of p beta}
\sum_{\substack{p\in\mathcal{P}\\ p\mid k}} p^{\beta}\le C_{\beta}\, k^{\beta}.
\end{equation}
\par Indeed, each $k$ can be factorized as
\[k=p_1^{\alpha_1}p_2^{\alpha_2}\cdots p_s^{\alpha_s}\ {\rm with\ primes}\ p_1<p_2<\dots<p_s\
{\rm and}\ \alpha_i\ge1 (1\le i\le s),\]
which indicates
\[\sum_{\substack{p\in\mathcal{P}\\p\mid k}} p^{\beta}=\sum_{i=1}^s p_i^{\beta}.\]
This, together with $k^{\beta}\ge(p_1 p_2\cdots p_s)^{\beta}$ since $\alpha_i\ge1$, implies
\begin{equation}\label{Rk} R(k):=\frac{1}{k^{\beta}}\sum\limits_{\substack{p\in\mathcal{P}\\p\mid k}} p^{\beta}
\le\frac{\sum_{i=1}^s p_i^{\beta}}{\prod_{j=1}^s p_j^{\beta}}
=\sum_{i=1}^s \prod_{\substack{j=1\\j\ne i}}^s p_j^{-\beta}.\end{equation}
Let $\{q_n\}_{n\ge1}$ label the sequence of all primes ordered increasingly.
 Since $p_n\ge q_n\ge n$ and $q_n\sim n\log n$ as $n\rightarrow\infty$ (see e.g.\,\cite{Hadamard1896}),
 by (\ref{Rk}) we derive
\[R(k) \le \sum_{i=1}^s \prod_{\substack{j=1\\j\ne i}}^s q_j^{-\beta}
= \frac{\sum_{i=1}^s q_i^{\beta}}{\prod_{j=1}^s q_j^{\beta}}\lesssim\frac{s^{\beta+1}(\log s)^{\beta}}{(s!)^{\beta}}.
\]
Combining this with $\lim\limits_{s\rightarrow\infty}\frac{s^{\beta+1}(\log s)^{\beta}}{(s!)^{\beta}}=0$
suggests that there must exist a constant $C_{\beta}>0$ depending only on $\beta$ such that $R(k)\le C_{\beta}$ for all $k$,
so the Claim (\ref{estimate of p beta}) holds.

\par Substituting (\ref{estimate of p beta}) into (\ref{weighted sum}) gains
\[W=\sum_{p\in\mathcal{P}} p^{\beta}S_p
\le\sum_{k=1}^{\infty}|a_k|\,C_{\beta}\,k^{\beta}=C_{\beta}M<\infty.\]
Thus $\sum_{p\in\mathcal{P}}\frac{1}{p}\big(p^{1+\beta}S_p\big)<\infty.$
Set $\lambda_p:=p^{1+\beta}S_p$, then $\sum_{p\in\mathcal{P}}\frac{1}{p}\lambda_p<\infty.$

\par Suppose Lemma \ref{positive density} is false. Then there exist
a subset $\mathcal{P}\subset\mathbb{P}$ with natural density $d_\mathbb{P}(\mathcal{P})>0$ and a constant $C>0$
such that the set
\[E_C:=\{ p\in\mathcal{P}:\ \sum_{p\mid k}|a_k|\le C\,p^{-\beta-1}\}=\{ p\in\mathcal{P}:\ \lambda_p\le C \}\]
has natural density $d_{\mathcal{P}}(E_C)<1$. So its complement
\[G_C :=\mathcal{P}\backslash E_C=\{ p\in\mathcal{P}:\ \lambda_p>C\}\]
has natural density $d_{\mathcal{P}}(G_C)>0$. Then by $d_\mathbb{P}(\mathcal{P})>0$, we have $d_{\mathbb{P}}(G_C)>0$.
Using Lemma \ref{reciprocal-sum} with $\mathcal{B}=\mathbb{P}$ attains
\[\sum_{p\in G_C} \frac{1}{p} = \infty,\]
and thus
\[\sum_{p\in G_C}\frac{1}{p}\lambda_p\ge C\sum_{p\in G_C}\frac{1}{p}=\infty,\]
which contradicts
\[\sum_{p\in G_C}\frac{1}{p}\lambda_p\le\sum_{p\in \mathcal{P}}\frac{1}{p}\lambda_p<\infty.\]
Therefore, Lemma \ref{positive density} follows.\end{proof}

\subsection{Proof of Theorem \ref{thm:lower}}
Now we restate and prove Theorem \ref{thm:lower} step by step.

\begin{theorem} {\rm(Theorem \ref{thm:lower})} \label{lower bound of dimension}
Let $g(x)=\sum_{k\in\mathbb{Z}}C_ke^{2\pi i k x}$ with $\sum_{k=1}^{\infty}|C_k|k^{\frac{\varepsilon+\delta}{2}}<\infty$
for some $\varepsilon>0$. If there exists a positive square-free
integer $d_0$ such that
\[\sum_{\substack{k=m^{2}d_0\\ m\in\mathbb{N}^+}}C_k\,k^{\delta/2}\neq0,\]
then
\[\underline{\dim}_{\mathrm{B}}\bigl(graph(G_\delta)\bigr)\ge\frac74-\frac{\delta}{2}.\]
\end{theorem}

\begin{proof}
\emph{{\rm Step 1:} Variation on quadratic residues.}
\par For a prime \(p\) and \(a\in R(p)\), where $R(p)$ is the set of quadratic residues mod $p$
as in (\ref{sum-Rp}), define
\[\Delta G_\delta\!\left(\frac{a}{p}\right)=G_\delta\!\left(\frac{a}{p}\right)
-G_\delta\!\left(\frac{a}{p}+\frac1{p^{2}}\right)\]
and
\[
S_k(p)=\sum_{n=1}^{\infty}\frac{1-e^{2\pi i k n^2/p^2}}{n^{1+\delta}}.
\]

\par\noindent Inserting the Fourier expansion of \(g\) gives
\begin{align*}
\Delta G_\delta\!\left(\frac{a}{p}\right)
&= \sum_{n=1}^\infty\frac{1}{n^{1+\delta}}\sum_{k\in\mathbb{Z}}C_k\bigl(e^{2\pi ikn^{2}a/p}-e^{2\pi ikn^{2}(a/p+1/p^{2})}\bigr)\\
&=\sum_{k\in\mathbb{Z}}C_k e^{2\pi ikn^{2}a/p}\sum_{n=1}^{\infty}\frac{1-e^{2\pi ikn^2/p^2}}{n^{1+\delta}}\\
&= \sum_{k\in\mathbb{Z}}C_k S_k(p)\,e^{2\pi ikn^{2}a/p}.
\end{align*}
By summing $\Delta G_\delta$ over all quadratic residues and using the Gauss sum evaluation we obtain
\[\sum_{a\in R(p)}\Delta G_\delta\!\left(\frac{a}{p}\right)
=\frac12\sum_{\substack{k\in\mathbb{Z}\\p\nmid k}}C_kS_k(p)\bigl(\epsilon_p\sqrt{p}\,\left(\frac{k}{p}\right)-1\bigr)
+\frac{p-1}{2}\sum_{\substack{k\in\mathbb{Z}\\p\mid k}}C_kS_k(p).\]

\emph{{\rm Step 2:} Isolating the main term.}
By Lemma~\ref{lem:error estimate}, \(S_k(p)=p^{-\delta}M_k + p^{-\delta}E_k(p)\). Substitute this into the expression above and split the sum
into the following four parts as
\[\sum_{a\in R(p)}\Delta G_\delta(a/p):= I_1 + I_2 + I_3 + I_4,\]
where
\[\begin{aligned}
I_1&=\frac12 p^{-\delta}\epsilon_p\sqrt{p}\sum_{\substack{k\in\mathbb{Z}\\p\nmid k}} C_kM_k\left(\frac{k}{p}\right),\quad\quad\quad\quad
I_2=-\frac12 p^{-\delta}\sum_{\substack{k\in\mathbb{Z}\\p\nmid k}} C_k M_k,\\
I_3&=\frac12 p^{-\delta}\sum_{\substack{k\in\mathbb{Z}\\p\nmid k}} C_kE_k(p)\bigl(\epsilon_p\sqrt{p}\,\left(\frac{k}{p}\right)-1\bigr),\quad
I_4=\frac{p-1}{2}\sum_{\substack{k\in\mathbb{Z}\\p\mid k}} C_k S_k(p).
\end{aligned}\]

\emph{{\rm Step 3:}  Estimating \(I_1\) via the Legendre symbol condition.}

Set\ $r_k:=C_k M_k$. For $k\geq 0$, change of variable \(u = t/\sqrt{k}\) shows that
\[M_k = \int_0^\infty \frac{1-e^{2\pi i k u^{2}}}{u^{1+\delta}}\,du
= k^{\delta/2}\int_0^\infty \frac{1-e^{2\pi i t^{2}}}{t^{1+\delta}}\,dt = k^{\delta/2} M_1.\]
 First $M_{-k}=\overline{M_k}$ implies $r_k = \overline{r_{-k}}$. On the other hand
 we have
\[\sum_{k\in\mathbb{Z}}|r_k|=2|M_1|\sum_{k=1}^{\infty}|C_k| k^{\frac{\delta}{2}}
\lesssim\sum_{k=1}^{\infty}|C_k| k^{\frac{\varepsilon+\delta}{2}}<\infty.\]

\par Furthermore, combining $r_k=C_kk^{\delta/2} M_1$ and the non-vanishing condition (\ref{nonvanishing}) we have
\(\sum_{\substack{k=m^{2}d_0, m\in\mathbb{N}^+}} r_k \neq 0\).
Thus by Lemma~\ref{lem:legendre-sum}, there exist
\(a,b\in\mathbb{N}^+\) with $b<a$ and a set of primes
\[\mathcal{P}=\{\text{odd\ prime}\ p:\ p\equiv b\,({\rm mod}\,a)\}\]
such that for all \(p\in\mathcal{P}\) we have
\[\Bigl|\sum_{k\in\mathbb{Z}} r_k\left(\frac{k}{p}\right)\Bigr|\ge \eta>0,\]
where \(\eta=\bigl(\frac{\sqrt{2}}{2}-\frac12\bigr)|R_{d_0}|\) and
\(R_{d_0}=\sum\limits_{\substack{k=m^{2}d_0\\ m\in\mathbb{N}^+}}C_k\,k^{\delta/2}\).
\par Also, $\left(\frac{k}{p}\right)=0$ if $p\mid k$, it follows that
\[|I_1| = \frac12 p^{1/2-\delta} \Bigl|\sum_{\substack{k\in\mathbb{Z}\\p\nmid k}}r_k\left(\frac{k}{p}\right)\Bigr|
=\frac12 p^{1/2-\delta} \Bigl|\sum_{k\in\mathbb{Z}}r_k\left(\frac{k}{p}\right)\Bigr|\gtrsim p^{1/2-\delta}.
 \]

\emph{{\rm Step 4:}\ Bounding the error terms.}
\par For \(I_2\), using $M_k = k^{\delta/2} M_1$ we have
\[|I_2|\le\frac12 p^{-\delta}\sum_{k\in\mathbb{Z}}|C_k M_k|\lesssim
p^{-\delta}\sum_{k=1}^{\infty}|C_k|k^{\frac{\varepsilon+\delta}{2}}\lesssim p^{-\delta}.\]
For $I_3$, applying Lemma \ref{lem:error estimate} with $0<\alpha<\min\{\varepsilon, 1\}$ gets
\begin{align*}
|I_3|
&\le\frac12 p^{1/2-\delta}\sum|C_k||E_k(p)|\lesssim
p^{1/2-\delta-\alpha}\sum_{k=1}^{\infty}|C_k|k^{\frac{\varepsilon+\delta}{2}}
\lesssim p^{\frac12 -\delta-\alpha}.
\end{align*}
\par For \(I_4\), since $\sum_{k=1}^{\infty}|C_k| k^{\frac{\varepsilon+\delta}{2}}<\infty$ and Lemma \ref{lem:Dirichlet} implies
$d_\mathbb{P}(\mathcal{P})=\frac{1}{\varphi(a)}>0$, applying Lemma \ref{positive density} with $a_k=C_k$,
$\beta=\frac{\varepsilon+\delta}{2}$ and $C=1$ obtains a subset $E_1\subseteq\mathcal{P}$
with $d_{\mathcal{P}}(E_1)>0$ such that
\[\sum_{p\mid k} |C_k| \le  p^{{-\frac{\varepsilon+\delta}{2}}-1}\]
for all $p\in E_1$. This, together with $|S_k(p)|\le2\zeta(1+\delta)$, yields
\[|I_4|\leq p\Big|\sum_{\substack{k\in\mathbb{Z}\\p\mid k}} C_k S_k(p)\Big|\lesssim p\sum_{p\mid k}|C_k|
\lesssim p^{{-\frac{\varepsilon+\delta}{2}}}.\]

\par Thus, combining the above estimates for all $p\in E_1$ and $0<\delta\leq1$ attains that
\begin{equation}\label{asymptotic lower bound}
\Bigl|\sum_{a\in R(p)}\Delta G_\delta(a/p)\Bigr|\gtrsim p^{1/2-\delta}.
\end{equation}

\emph{{\rm Step 5:} From variation to dimension.}
\par Let $\mathcal{P}$ and $E_1$ be the same as in Step 3 and 4.
Fix a large integer \(N\) and choose  primes \(p\in E_1\) with \(p\asymp N^{1/2}\).
As $N\rightarrow\infty$, by Lemma \ref{lem:Dirichlet} we have
\[{\rm card}\,\bigl\{p\in E_1:\ p\asymp N^{1/2}\bigr\}\sim \frac{N^{1/2}}{\varphi(a)\log(N^{1/2})}.\]
Recall that $R(p)$ consists of quadratic residues ${\rm mod}\ p$.
The interval \([a/p,\,a/p+1/p^{2}]\) has length \(1/p^{2}\sim N^{-1}\) for each \(a\in R(p)\),
by Lemma \ref{lem:fip} and (\ref{asymptotic lower bound}) we have
\begin{align*}
\underline{\dim}_{\mathrm{B}}(\operatorname{graph}(G_{\delta}))
&\geq \varliminf_{N \to\infty} \frac{\log\left( N\sum\limits_{p\in\mathcal{P}}\sum\limits_{1\leq a\leq p-1}|\Delta G_\delta(a/p)|\right)}{\log N}\\
&\geq \varliminf_{N \to\infty} \frac{\log\left( N\sum\limits_{p\in E_1}\Big|\sum\limits_{a\in R(p)}
\Delta G_\delta(a/p)\Big|\right)}{\log N}\\
&\geq \varliminf_{N\to\infty}\frac{\log\left( N\sum\limits_{p\in E_1}p^{\frac12-\delta}\right)}{\log N}\\
&=\frac74-\frac\delta2.
\end{align*}
\end{proof}
\section{Upper bound for the upper box dimension}\label{sec:upper}
In this section, at first we analyze convergence of a series on the Fourier coefficients of $g$.
Secondly, we try to obtain asymptotic upper bound for partial exponential sums concerning $g$ by Weyl's differencing.
Based on this, we give the proof of Theorem \ref{thm:upper}.
\subsection{Preparatory lemmas}
\begin{lemma}\label{Fourier-coefficient}
Let $g(x)$ be a Lipschitz continuous periodic function on $\mathbb{R}$ with period 1 and Fourier expansion
\[g(x)=\sum_{k\in\mathbb{Z}}C_k e^{2\pi ikx}.\]
Then \[\sum_{k\in\mathbb{Z}}|C_k|^2 k^2<\infty.\]
\end{lemma}
\begin{proof}
Since a Lipschitz continuous function is absolutely continuous, its derivative $g'(x)$ exists almost everywhere, and for almost all $x$ where
$g'(x)$ exists, we have \[|g'(x)| \leq L,\] where $L$ is the Lipschitz constant of $g(x)$.
\par Indeed, since $g(x)$ is Lipschitz,
\[\left|\frac{g(x+h)-g(x)}{h}\right| \leq \frac{Lh}{h}=L.\]
Thus,\[|g'(x)|=\left|\lim_{h \to 0} \frac{g(x+h)-g(x)}{h}\right| = \lim_{h \to 0} \left|\frac{g(x+h) - g(x)}{h}\right| \leq L.\]
It follows that $g'(x)\in L^\infty[0,1]$ and $|g'(x)|\leq L$ almost everywhere. Hence, $g'(x) \in L^2[0,1]$.

\par Write the Fourier expansion of $g'(x)$ as
\[g'(x)=\sum_{k\in\mathbb{Z}} \widetilde{C}_k e^{2\pi i k x}.\]
Integration by parts shows
\begin{align*}
\widetilde{C}_k &= \int_0^1 g'(x) e^{-2\pi i k x}\, dx \\
&=\left. e^{-2\pi i k x} g(x)\right|_0^1 + 2\pi i k \int_0^1g(x)e^{-2\pi ikx}\, dx \\
&=2\pi ikC_k.
\end{align*}
\par By Parseval identity, we get
\[\int_0^1 |g'(x)|^2\, dx=\sum_{k \in \mathbb{Z}} |\widetilde{C}_k|^2.\]
Substituting $\widetilde{C}_k = 2\pi i k C_k$ yields
\[\int_0^1 |g'(x)|^2 \, dx = \sum_{k \in \mathbb{Z}} |2\pi i k C_k|^2 = 4\pi^2 \sum_{k \in \mathbb{Z}} k^2 |C_k|^2 < \infty.\]
Therefore,\[\sum_{k \in \mathbb{Z}} |C_k|^2 k^2<\infty.\]
\end{proof}
\begin{lemma}\label{weyl}
Denote $e(x)=e^{2\pi ix}$ for $x\in\mathbb{R}$.
Let $M, N, P$ be positive integers, $1\leq a<p$, $\gcd(a, p)=1$, and $\left|x - {a}{p}\right|\leq \frac{1}{p^2}$. Then
\[\left| \sum_{n=M}^{M+N-1} e(k n^2 x)\right|\lesssim\sqrt{k}\left(\frac{N}{\sqrt{p}}+\sqrt{N\log p}\right)+\sqrt{p\log p}.\]
\end{lemma}
\begin{proof}
Let $S_k=\sum\limits_{n=M}^{M+N-1} e(k n^2 x)$. By Weyl's differencing, we obtain
\begin{eqnarray} \label{Sk2}
|S_k|^2 &=&\sum_{n=M}^{M+N-1} \sum_{m=M}^{M+N-1} e\left(km^2x-kn^2x\right)\nonumber \\
&=&\sum_{n=M}^{M+N-1}\sum_{h=M-n}^{M+N-1-n} e\left(k(n+h)^2x-kn^2x\right)\nonumber  \\
&=&N+2{\rm Re}\sum_{h=1}^{N-1}\sum_{n=M}^{M+N-h-1} e\left(k(n+h)^2x-kn^2x\right)\nonumber \\
&\lesssim& N + \sum_{h=1}^{N-1} \left| \sum_{n=M}^{M+N-h-1} e\left(k(n+h)^2x-kn^2x\right) \right|\nonumber \\
&=&N + \sum_{h=1}^{N-1} \left| \sum_{n=M}^{M+N-h-1} e(2khxn) \right|\nonumber \\
&\lesssim& N + \sum_{h=1}^{N-1} \min\left(N, \frac{1}{\|2khx\|}\right).
\end{eqnarray}

\par Let $m = 2k h$, then
\begin{equation} \label{minNkhx}\sum_{h=1}^{N-1}\min\left(N, \frac{1}{\|2khx\|}\right)
\lesssim\sum_{m=1}^{2Nk} \min\left(N, \frac{1}{\|m x\|}\right).\end{equation}
Applying (9) in \cite[Chapter 3]{Montgomery1994} for $H=2Nk$ results in
\[\sum_{m=1}^{2Nk}\min\left(N, \frac{1}{\|m x\|}\right)\lesssim\frac{kN^2}{p}+Nk\log p+N+p\log p.\]
\par Substituting this into (\ref{Sk2}) and (\ref{minNkhx}) leads to
\[|S_k|^2\lesssim\frac{kN^2}{p}+Nk\log p+N+p\log p,\]
which implies
\[|S_k|\lesssim\sqrt{k}\left(\frac{N}{\sqrt{p}}+\sqrt{N\log p}\right)+\sqrt{p\log p}.\]
\end{proof}

\par The next lemma gives a non-trivial estimate for incomplete sums of \(g(n^{2}x)\) when \(x\) is near a rational.
\begin{lemma}\label{gn2x-a/p}
Let $g(x)$ be a Lipschitz continuous function on $\mathbb{R}$ with period 1.
Let $M, N$ and $P$ be positive integers, $1\leq a<p$, $\gcd(a, p)=1$, and $\left|x-\frac{a}{p}\right|\leq\frac{1}{p^2}$.
Then\[\left|\sum_{n=M}^{M+N-1}g(n^2 x)\right|\lesssim\frac{N}{\sqrt{p}}\sqrt{\log p}+\sqrt{N}\log p+\sqrt{p}\log p.\]
\end{lemma}

\begin{proof}
Let $S_k(x)=\sum\limits_{n=M}^{M+N-1} e(k n^2 x)$. Then
\begin{equation} \label{gn2Sk}
\left|\sum_{n=M}^{M+N-1}g(n^2 x)\right|=\left|\sum_{n=M}^{M+N-1}\sum_{k\in \mathbb{Z}}C_k e(kn^2x) \right|
=\left|\sum_{k\in\mathbb{Z}}C_kS_k(x)\right|\leq 2\sum_{k=1}^{\infty} |C_k|\cdot|S_k(x)|.
\end{equation}
Utilizing Lemma \ref{weyl} gets
\[|S_k(x)|\lesssim\min\left(N, \sqrt{k} \left( \frac{N}{\sqrt{p}}+\sqrt{N\log p}\right)+\sqrt{p\log p}\right),\]
inserting this into (\ref{gn2Sk}) yields
\begin{equation}\label{gn2Ck}
\left|\sum_{n=M}^{M+N-1}g(n^2x)\right|\lesssim\left(\frac{N}{\sqrt{p}}+\sqrt{N\log p}+\sqrt{p\log p}\right)\sum_{k=1}^p |C_k|
\sqrt{k}+N\cdot\sum_{k=p+1}^{\infty}|C_k|.
\end{equation}
By Lemma \ref{Fourier-coefficient}, we have
\begin{equation}\label{Ck-sqrtk}
\sum_{k=1}^p|C_k|\sqrt{k}=\sum_{k=1}^p\frac{|C_k| k}{\sqrt{k}}\leq \left(\sum_{k=1}^p|C_k|^2 k^2\right)^{1/2}
\left( \sum_{k=1}^p\frac{1}{k}\right)^{1/2}\lesssim\sqrt{\log p}. \end{equation}

\par Also, applying Cauchy-Schwarz inequality gains
\begin{equation}\label{Ck}
\sum_{k=p+1}^{\infty}|C_k|=\sum_{k=p+1}^{\infty} \frac{|C_k| k}{k}
\leq \left(\sum_{k=p+1}^{\infty}|C_k|^2 k^2 \right)^{1/2}\left(\sum_{k=p+1}^{\infty}\frac{1}{k^2}\right)^{1/2}
\lesssim\frac{1}{\sqrt{p}}.\end{equation}
Substituting (\ref{Ck-sqrtk}) and (\ref{Ck}) into (\ref{gn2Ck}) results in
\[\left|\sum_{n=M}^{M+N-1}g(n^2 x)\right|\lesssim\frac{N}{\sqrt{p}}\sqrt{\log p}+\sqrt{N}\log p+\sqrt{p}\log p,\]
which gives the desired estimate.
\end{proof}

\subsection{Proof of Theorem \ref{thm:upper}}
We now turn to restate and prove Theorem \ref{thm:upper} as follows.
\begin{theorem}{\rm(Theorem \ref{thm:upper})}
Let \(g:\mathbb{R}\to\mathbb{R}\) be a {\rm 1}-periodic function. If $g'$ is Lipschitz continuous, then
\[\overline{\dim}_{\mathrm{B}}\bigl(graph(G_\delta)\bigr)\le\frac74-\frac{\delta}{2}.\]
\end{theorem}

\begin{proof}
\noindent Fix a large integer \(N\) and consider the Farey dissection $\{J_{a/p}\}$ of order \(N^{1/2}\).
Let \(\lambda(a/p;N)\) be the number of grid squares of side \(1/N\) needed to cover
\(graph(G_\delta)\cap(J_{a/p}\times\mathbb{R})\),
 and denote by $\lambda\left(\frac{a}{p};\, N,\, k\right)$ the number of
 those squares contained inside the strip $I_k \times \mathbb{R}$,
 where $I_k = \left[\frac{k}{N},\, \frac{k+1}{N}\right]$.

\par Define
\[S_n(x)=\sum_{k=1}^n g(k^2 x),\qquad \widetilde{S}_n(x)=\sum_{k=1}^n g'(k^2 x).\]
Then, given any $\varepsilon>0$, taking  $M=1$ and $N=n$ in Lemma \ref{gn2x-a/p}
derives
\begin{equation}\label{gSnx}
|S_n(x)|\lesssim\frac{n}{\sqrt{p}}\sqrt{\log p}+\sqrt{n}\log p+\sqrt{p} \log p
\lesssim_\varepsilon p^\varepsilon \left( \frac{n}{\sqrt{p}} + \sqrt{n} + \sqrt{p}\right)
\end{equation}  for all $x$ with $|x - a/p|\le 1/p^2$.
Replacing $g$ with $g^{\prime}$ in (\ref{gSnx}) produces
\begin{equation}\label{g'Snx}
|\widetilde{S}_n(x)|\lesssim_\varepsilon p^\varepsilon\left(\frac{n}{\sqrt{p}}+\sqrt{n}+\sqrt{p}\right).
\end{equation}
\par By the definition of $\lambda(a/p; N, k)$ and mean value theorem, there exist $\xi, x_0, y_0 \in I_k \cap J_{\frac{a}{p}}$ such that
\begin{eqnarray}\label{lambdaap}
\lambda\left(\frac{a}{p}; N, k\right)&\le& 2+N\cdot\sup_{x,y\in I_k\cap J_{\frac{a}{p}}}|G_{\delta}(x)-G_{\delta}(y)|\nonumber \\
&=&2+N|G_{\delta}(x_0)-G_{\delta}(y_0)|\nonumber \\
&\le&2+N\left|\sum_{n\le H}n^{1-\delta}g'(n^2 \xi)\right|+N\left|\sum_{n>H}\frac{1}{n^{1+\delta}} \bigl(g(n^2 x_0)-g(n^2y_0)\bigr)\right|,
\end{eqnarray}
where $H$ is a  positive integer to be chosen later.

\par Using Abel summation formula and (\ref{g'Snx}) obtains
\begin{eqnarray}\label{n1-deltag'}
\left|\sum_{n\le H} n^{1-\delta} g'(n^2 \xi)\right|
&\le&\left|\sum_{n\le H}\bigl[n^{1-\delta}-(n+1)^{1-\delta}\bigr]\widetilde{S}_n(\xi)\right|+H^{1-\delta}|\widetilde{S}_H(\xi)|\nonumber \\
&\lesssim_\varepsilon& p^\varepsilon\sum_{n\le H}\bigl[(n+1)^{1-\delta}-n^{1-\delta} \bigr] \left(\frac{n}{\sqrt{p}}+\sqrt{n}+\sqrt{p}\right)\nonumber\\
&\qquad& +p^\varepsilon H^{1-\delta}\left(\frac{H}{\sqrt{p}}+\sqrt{H}+\sqrt{p}\right)\nonumber \\
&\lesssim_\varepsilon& p^\varepsilon \left(H^{2-\delta}p^{-\frac{1}{2}}+H^{\frac{3}{2}-\delta}+H^{1-\delta}\sqrt{p}\right).
\end{eqnarray}
Applying the same argument as (\ref{n1-deltag'}) gets
\begin{eqnarray}\label{n-1-deltag}
&&\left|\sum_{n>H} \frac{1}{n^{1+\delta}}\bigl( g(n^2 x_0)-g(n^2 y_0)\bigr)\right|\nonumber \\
&\le&\left(\sum_{n>H}\bigl[n^{-1-\delta}-(n+1)^{-1-\delta}\bigr]\bigl(|S_n(x_0)|+|S_n(y_0)|\bigr)\right)\nonumber \\
&\qquad& +(H+1)^{-1-\delta} \bigl(|S_H(x_0)|+|S_H(y_0)|\bigr)\nonumber \\
&\lesssim_\varepsilon& p^\varepsilon \left( H^{-\delta} p^{-\frac{1}{2}}+H^{-\frac{1}{2}-\delta}+H^{-1-\delta}\sqrt{p}\right).
\end{eqnarray}
Taking $H=[\sqrt{N}]$ and substituting (\ref{n1-deltag'}) and (\ref{n-1-deltag})
into (\ref{lambdaap}) lead to
\[\lambda\left(\frac{a}{p}; N, k\right)\lesssim_\varepsilon p^\varepsilon \left( N^{1-\frac{\delta}{2}} p^{-\frac{1}{2}} + N^{\frac{3}{4}-\frac{\delta}{2}} + N^{\frac{1}{2}-\frac{\delta}{2}}\sqrt{p}\right).\]
\par Since the number of strips \(I_k\) intersecting \(J_{a/p}\) is
at most \(N|J_{a/p}|+2\) and \(|J_{a/p}|\lesssim1/(pN^{1/2})\),
\begin{align*}
\lambda\left(\frac{a}{p}; N\right)
&\lesssim_\varepsilon \frac{|J_{\frac{a}{p}}|}{N^{-1}} p^\varepsilon \left( N^{1-\frac{\delta}{2}} p^{-\frac{1}{2}} + N^{\frac{3}{4}-\frac{\delta}{2}} + N^{\frac{1}{2}-\frac{\delta}{2}} \sqrt{p} \right)\\
&\lesssim_\varepsilon p^\varepsilon \left( N^{\frac{3}{2}-\frac{\delta}{2}} p^{-\frac{3}{2}} + N^{\frac{5}{4}-\frac{\delta}{2}} p^{-1} + N^{1-\frac{\delta}{2}}p^{-\frac{1}{2}}\right).
\end{align*}
Hence, the number of squares in the grid $\mathcal{M}_N$ needed to cover the graph of $G_{\delta}(x)$ has the upper bound
\begin{align*}
&\sum_{p\le N^{\frac{1}{2}}} \sum_{(a,p)=1} p^\varepsilon \left( N^{\frac{3}{2}-\frac{\delta}{2}} p^{-\frac{3}{2}}+
N^{\frac{5}{4}-\frac{\delta}{2}} p^{-1} + N^{1-\frac{\delta}{2}} p^{-\frac{1}{2}} \right)\\
&\lesssim_\varepsilon \sum_{p\le N^{\frac{1}{2}}} p^{1+\varepsilon} \left( N^{\frac{3}{2}-\frac{\delta}{2}} p^{-\frac{3}{2}} + N^{\frac{5}{4}-\frac{\delta}{2}} p^{-1} + N^{1-\frac{\delta}{2}} p^{-\frac{1}{2}} \right)\\
&\lesssim_\varepsilon N^{\frac{7}{4}-\frac{\delta}{2}+\frac{\varepsilon}{2}}.
\end{align*}
\par Consequently, it follows that
\[\underline{\dim}_{\mathrm{B}}(\text{graph}(G_\delta)) \le \limsup_{N\to\infty}
\frac{\log\left( N^{\frac{7}{4}-\frac{\delta}{2}+\frac{\varepsilon}{2}}\right)}
{\log N}=\frac{7}{4}-\frac{\delta}{2}+\frac{\varepsilon}{2}.\]
Since \(\epsilon>0\) is arbitrary, we can conclude
\[\underline{\dim}_{\mathrm{B}}(\text{graph}(G_{\delta})) \le \frac{7}{4} - \frac{\delta}{2}.\]
\end{proof}

\section{The critical case: vanishing mechanism and its complexity}\label{sec:critical}

The lower bound in Theorem~\ref{thm:lower} hinges on the non-vanishing condition
\[\sum_{k=m^{2}d_0, m\in\mathbb{N}^+}C_k k^{\delta/2}\neq0.\] A natural
question arises: what happens when these sums vanish for any given square-free $d_0\in\mathbb{N}^+$?
For a fixed square-free $d_0$, if this sum vanishes, we call it the \emph{critical case};
otherwise, the \emph{uncritical case}.

\subsection{Illustrative examples}

We now examine two cases where the vanishing condition holds, demonstrating the richness of the critical case.
\begin{example}\label{ex:translation}
Let $g(x)=2^{\delta+1}\cos(2\pi x)-2\cos(8\pi x)$ {\rm(}with $0<\delta\le1${\rm)}.
Then $C_{\pm1}=2^{\delta}$, $C_{\pm4}=-1$, $C_k=0$ for $k\not\in\{\pm1,\pm4\}$.  A
direct computation shows for $d_0=1$ that
\[\sum_{\substack{k=m^{2}\\ m\in\mathbb{N}^+}}C_k k^{\delta/2}=C_1\cdot1^{\delta/2}+C_4\cdot4^{\delta/2}=2^{\delta}-2^{\delta}=0,
\]
so the condition of Theorem {\rm \ref{thm:lower}} fails.  Nevertheless, using the identity
\[G_{\delta}(x)=2^{\delta+1}F_{\delta}(x)-2F_{\delta}(4x) = -2^{\delta+1}F_{\delta}(x+\tfrac12),\]
where $F_{\delta}(x)=\sum_{n=1}^{\infty}\frac{\cos(2\pi n^{2}x)}{n^{1+\delta}}$, we see that $graph(G_{\delta})$ is merely a translate of
$graph(F_{\delta})$, hence its dimension is still $\frac{7}{4}-\frac{\delta}{2}$.  Thus the vanishing condition is not sufficient for
collapse; additional symmetries may preserve the dimension.
\end{example}

\begin{example}\label{ex:elliptic}
Let $E/\mathbb{Q}$ be a modular elliptic curve with Hasse--Weil $L$-function
$L(s,E)=\sum\limits_{n=1}^{\infty}a_{n}n^{-s}$, and consider
\[ F_{\delta}(x)=\sum_{n=1}^{\infty}\frac{a_{n}}{n^{1+\delta}}\cos(2\pi nx), \qquad 0<\delta<1.\]
Chamizo {\rm\cite{Chamizo2004}} proved that
$\dim_{\mathrm{B}}(\operatorname{graph}(F_{\delta}))=2-\delta$ for $0<\delta<1$.
We show that when $F_{\delta}$ is written in the Riemann-function form
$F_{\delta}(x)=\sum_{n=1}^{\infty}n^{-(1+\delta)}g(n^{2}x)$ for an appropriate
$1$-periodic function $g$ with Fourier coefficients $C_{k}$, the vanishing
condition
\[\sum_{\substack{k=m^{2}d,\\ m\in\mathbb{Z}}}C_k k^{\delta/2}=0\]
holds for every square-free positive integer $d$.

\par Expanding $g(n^{2}x)=\sum\limits_{k\in\mathbb{Z}}C_{k}e^{2\pi ikn^{2}x}$ and
interchanging summation {\rm(}justified below{\rm)} yields
\[F_{\delta}(x)=\sum_{l\in\mathbb{Z}}\!\left(\sum_{\substack{n\geq 1\\n^{2}\mid l}}
\frac{C_{l/n^{2}}}{n^{1+\delta}}\right)e^{2\pi ilx}.\]
Comparing with the Fourier expansion $\widehat{F}_{\delta}(l)=a_{|l|}/(2|l|^{1+\delta})$
for $l\neq 0$ gives
\begin{equation}\label{eq:FourierID}
\sum_{\substack{n\geq 1\\n^{2}\mid l}}\frac{C_{l/n^{2}}}{n^{1+\delta}}
=\frac{a_{l}}{2l^{1+\delta}}, \qquad l\geq 1.
\end{equation}
Fix a square-free integer $d\geq 1$ and write $l=dm^{2}$. Since $d$ is
square-free, $n^{2}\mid dm^{2}$ if and only if $n\mid m$. Setting $q=m/n$,
identity~\eqref{eq:FourierID} becomes
\[\sum_{q\mid m}\frac{C_{dq^{2}}}{(m/q)^{1+\delta}}
=\frac{a_{dm^{2}}}{2(dm^{2})^{1+\delta}}.\]
Multiplying both sides by $m^{1+\delta}$ and writing
$h(k):=a_{dk^{2}}/(2k^{1+\delta})$,
we obtain the Dirichlet convolution
$\sum\limits_{q\mid m}C_{dq^{2}}\,q^{1+\delta}=h(m)/d^{1+\delta}$,
and M\"{o}bius inversion formula {\rm(Lemma \ref{Mobius})} gives
\begin{equation}\label{eq:Mobius}
C_{dm^{2}}\,m^{1+\delta}
=\frac{1}{d^{1+\delta}}\sum_{k\mid m}\mu\!\left(\tfrac{m}{k}\right)h(k).
\end{equation}

\smallskip
The Hasse bound $|a_{n}|\lesssim_{\varepsilon}n^{1/2+\varepsilon}$ gives
$|a_{dk^{2}}|\lesssim_{\varepsilon}d^{1/2+\varepsilon}k^{1+2\varepsilon}$, hence
\begin{equation}\label{eq:hbound}
|h(k)|\lesssim_{\varepsilon} d^{1/2+\varepsilon}\,k^{2\varepsilon-\delta}.
\end{equation}
For fixed $d$, choosing $\varepsilon<\delta/2$ makes the exponent
$2\varepsilon-\delta-1<-1$, so
\begin{equation}\label{eq:hkbound}
  \sum_{k=1}^{\infty}\frac{|h(k)|}{k} < \infty.
\end{equation}
From~\eqref{eq:Mobius} and~\eqref{eq:hbound}, using $k^{2\varepsilon-\delta}\leq 1$
for $k\leq m$ when $\varepsilon<\delta/2$,
\[|C_{dm^{2}}|\lesssim_{\varepsilon}\frac{d^{\varepsilon-1/2-\delta}\,\tau(m)}{m^{1+\delta}},\]
where $\tau$ is the number-of-divisors function.
Since $\sum_{m=1}^{\infty}\tau(m)/m^{1+\delta}<\infty$ for $\delta>0$,
the series $\sum_{m=1}^{\infty}|C_{dm^{2}}|$ converges absolutely,
justifying the interchange of summation in the derivation above.

\smallskip
\noindent\textit{Vanishing of $S$.}
Substituting~\eqref{eq:Mobius} into $S$ and using
$(dm^{2})^{\delta/2}=d^{\delta/2}m^{\delta}$ yields
\[S=\frac{1}{d^{1+\delta/2}}\sum_{m=1}^{\infty}
\frac{1}{m}\sum_{k\mid m}\mu\!\left(\tfrac{m}{k}\right)h(k).\]
It remains to show that the right-hand side vanishes. Let
$M(t):=\sum_{l\leq t}\mu(l)/l$, then $M(t)\to 0$ as $t\to\infty$, and thus $C:=\sup_{t\geq 1}|M(t)|<\infty$,
which is equivalent to the prime number theorem {\rm(see e.g. \cite[Theorem 4.16]{Apostol1976})}.
Define the partial sum
\[T_N:=\sum_{m=1}^{N}\frac{1}{m}\sum_{k\mid m}\mu\!\left(\tfrac{m}{k}\right)h(k)
=\sum_{kl\leq N}\frac{h(k)}{k}\cdot\frac{\mu(l)}{l},\]
where the second equality is the substitution $l=m/k$, valid termwise.
Grouping over $l$ for each fixed $k$,
\[T_N = \sum_{k=1}^{N}\frac{h(k)}{k}\,M\!\left(\lfloor N/k\rfloor\right).\]
Given $\varepsilon>0$, choose $K$ such that $\sum\limits_{k>K}|h(k)|/k<\varepsilon$,
which is possible by~\eqref{eq:hkbound}. Then
\[ |T_N|\leq\sum_{k=1}^{K}\frac{|h(k)|}{k}\,\Bigl|M\!\Bigl(\lfloor\tfrac{N}{k}\rfloor\Bigr)\Bigr|
  +\,C\sum_{k>K}\frac{|h(k)|}{k}
  \leq \sum_{k=1}^{K}\frac{|h(k)|}{k}\,\Bigl|M\!\Bigl(\lfloor\tfrac{N}{k}\rfloor\Bigr)\Bigr|
  +\,C\varepsilon.\]
As $N\to\infty$, each $M(N/k)\to 0$ for fixed $k$, so the first sum tends
to $0$. Since $\varepsilon>0$ was arbitrary, $T_N\to 0$, and therefore $S=0$.

\par This example shows that the vanishing condition $S=0$ alone does not
determine the box dimension: indeed,
$\dim_{\mathrm{B}}(\operatorname{graph}(F_{\delta}))=2-\delta$, which is
strictly less than $\frac{7}{4}-\frac{\delta}{2}$ for $\frac{1}{2}<\delta<1$
and strictly greater than $\frac{7}{4}-\frac{\delta}{2}$ for
$0<\delta<\frac{1}{2}$, yet $S=0$ holds in both cases.
\end{example}

\par These examples illustrate that the vanishing condition can arise from
deep number theory (e.g., the Prime Number Theorem and $L$-functions) and
often indicates a reduction in fractal complexity.  They also show that the relation between the arithmetic sums and the dimension is subtle:
Example \ref{ex:translation} shows that additional symmetries can override its effect.
Nevertheless, in many natural situations vanishing is a strong indicator of dimension collapse.

\par Similarly,  denote for $0<\delta\le1$ that
\[F_{\delta}(x)=\sum_{n=1}^{\infty}\frac{f(n^{2}x)}{n^{1+\delta}}.\] For readers'
intuitive understanding complexity of the critical cases and comparing them with uncritical cases, some specific examples of
$graph(G_{\delta})$ and $graph(G_{\delta}, F_{\delta})$ are shown at the end (see Figures 1-6).

\section{Conclusion and further problems}\label{sec:conclusion}

We have established sharp box dimension bounds for generalized Riemann-type functions under mild conditions on $g$,
extending the classical results of Chamizo and C\'ordoba. A new arithmetic mechanism is identifies the vanishing of weighted
sums over square-free integers, which governs potential dimension collapse. Examples illustrate that this condition is intimately
linked to the Prime Number Theorem and analytic properties of $L$-functions.
\par Several questions remain open:
\begin{itemize}
  \item In the non-vanishing case, does the lower bound always give the exact dimension?  By Theorem \ref{thm:upper},
  this is indeed the case if the seed function $g$ satisfies certain regularity condition, but a general case remains unknown.
  \item Can one characterise the precise fractal dimension when the vanishing condition holds but the graph is not trivial (as in
      Example~\ref{ex:translation}). Specially, does there exist a trigonometric polynomial $g$ satisfying the vanishing condition and
      $\dim_{\mathrm{B}}(graph(G_{\delta}))<\frac74 -\frac{\delta}{2}$?
  \item What is the role of the Riemann $\zeta$-function and $L$-functions in determining the dimension ?  Example~\ref{ex:elliptic} points to
      deep connections with arithmetic geometry.
  \item To what extent can the regularity conditions on $g$ in Theorem~\ref{thm:lower} and Theorem~\ref{thm:upper} be weakened?
\end{itemize}
We hope that the ideas presented here will stimulate further research at the fertile interface between number theory and fractal geometry.


\begin{figure}\label{Graph1}
\centering
\includegraphics[width=6.5 in, height=3.1 in]{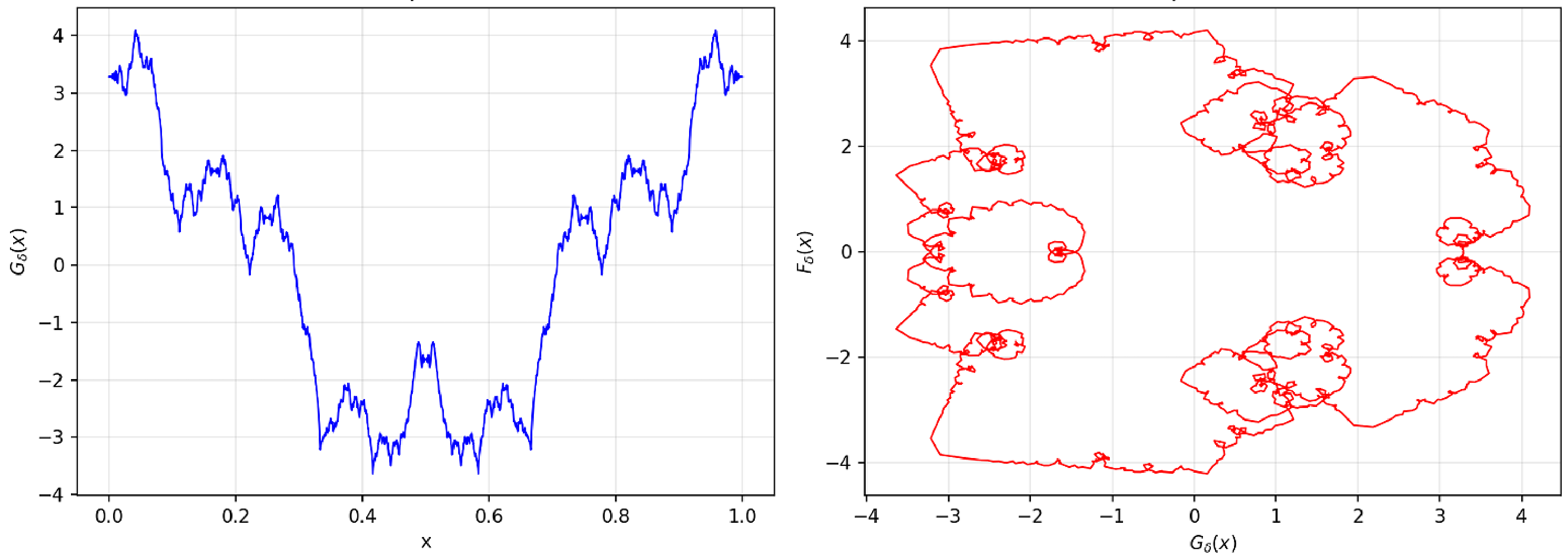}
\caption{$graph(G_{\delta})$ (in blue)
and $graph\big((G_{\delta}, F_{\delta})\big)$ (in red) with $\delta=1, g(x)=3\cos(2\pi x)-\cos(18\pi x)$ and
$f(x)=3\sin(2\pi x)-\sin(18\pi x)$ of critical case.}
\end{figure}

\begin{figure} \label{Graph2}
\centering
\includegraphics[width=6.5 in, height=3.1 in]{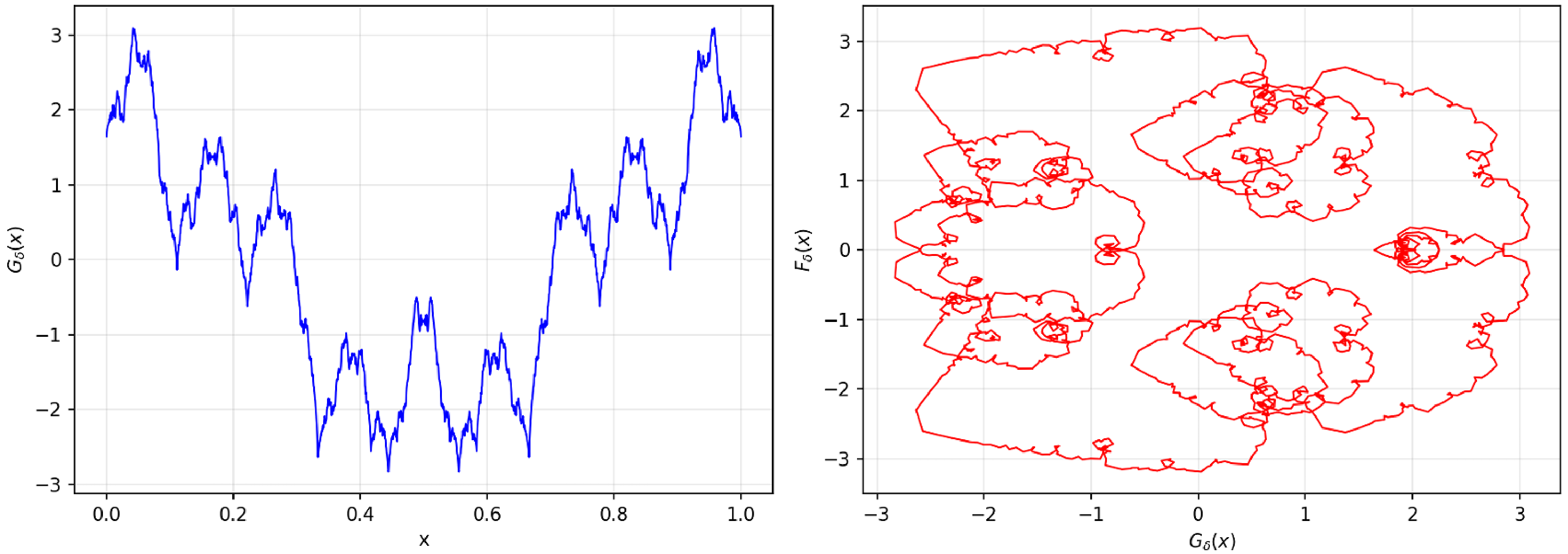}
\caption{$graph(G_{\delta})$ (in blue)
and $graph\big((G_{\delta}, F_{\delta})\big)$ (in red) with $\delta=1$, $g(x)=2\cos(2\pi x)-\cos(18\pi x)$ and
$f(x)=2\sin(2\pi x)-\sin(18\pi x)$ of uncritical case.}
\end{figure}

\begin{figure} \label{Graph3}
\centering
\includegraphics[width=6.5 in, height=3.1 in]{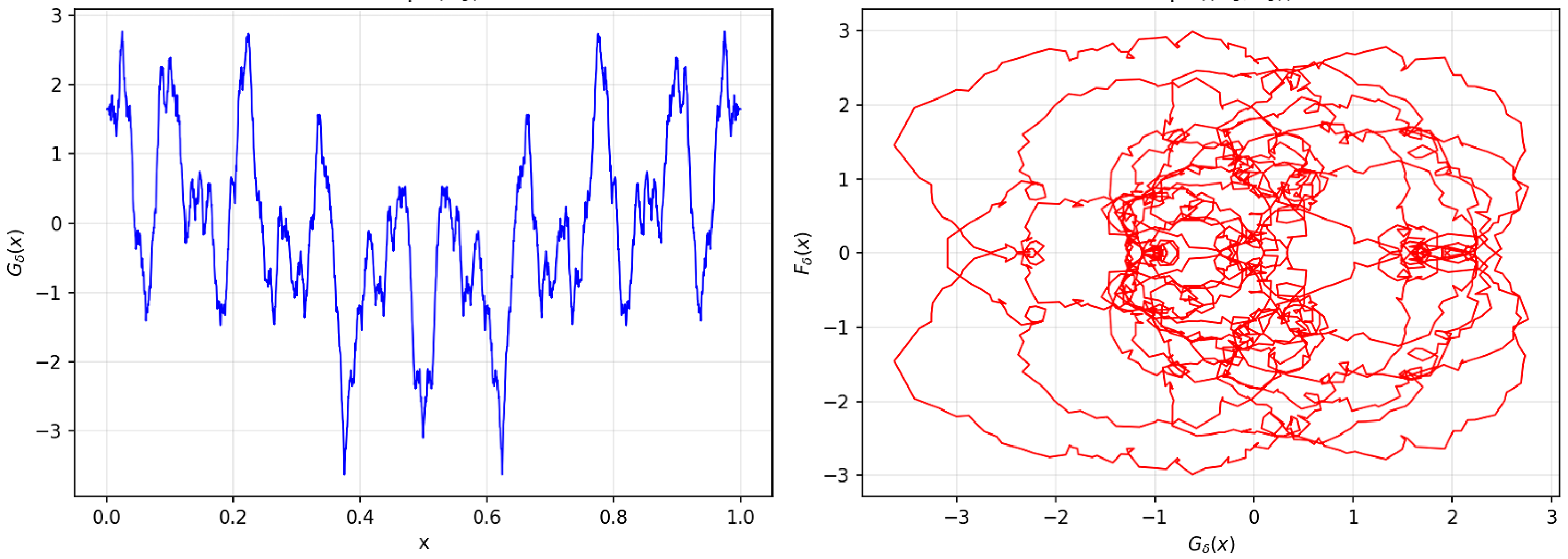}
\caption{$graph(G_{\delta})$ (in blue)
and $graph\big((G_{\delta}, F_{\delta})\big)$ (in red) with $\delta=1, g(x)=\cos(2\pi x)+\cos(18\pi x)
-\cos(32\pi x)$ and $f(x)=\sin(2\pi x)+\sin(18\pi x)-\sin(32\pi x)$ of critical case.}
\end{figure}

\begin{figure} \label{Graph4}
\centering
\includegraphics[width=6.5 in, height=3.1 in]{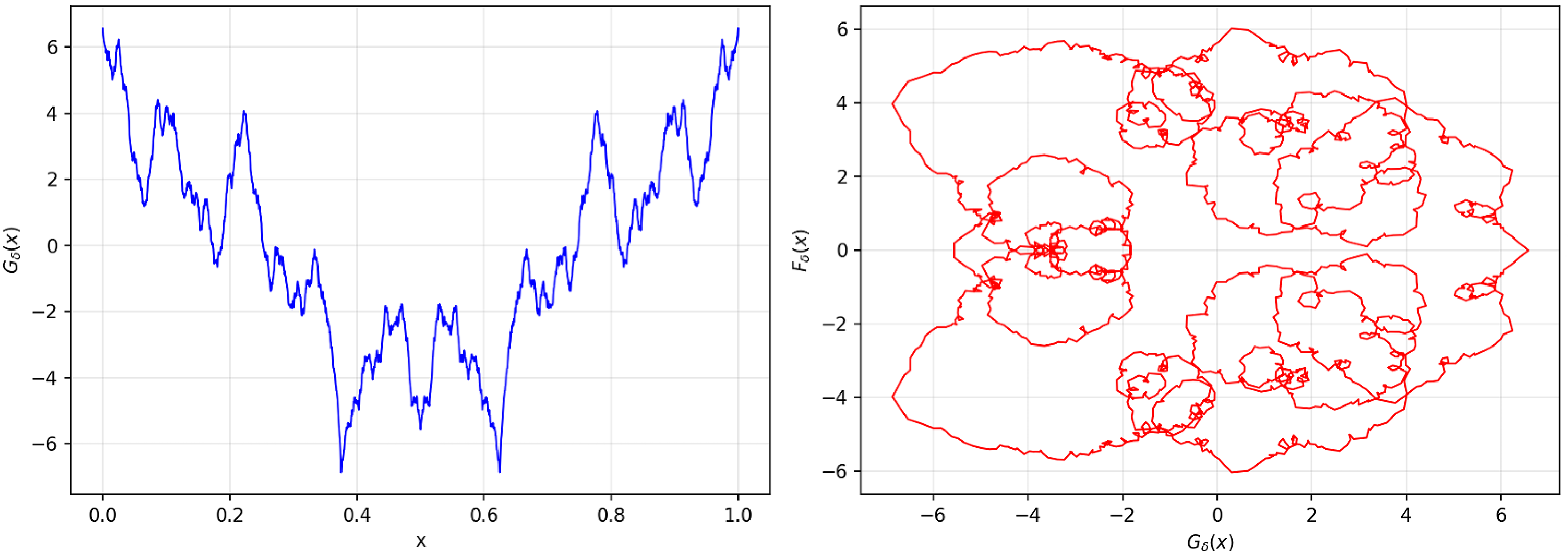}
\caption{$graph(G_{\delta})$ (in blue)
and $graph\big((G_{\delta}, F_{\delta})\big)$ (in red) with $\delta=1$ and $g(x)=4\cos(2\pi x)+\cos(18\pi x)
-\cos(32\pi x), f(x)=4\sin(2\pi x)+\sin(18\pi x)-\sin(32\pi x)$ of uncritical cases.}
\end{figure}

\begin{figure} \label{Graph5}
\centering
\includegraphics[width=6.5 in, height=3.1 in]{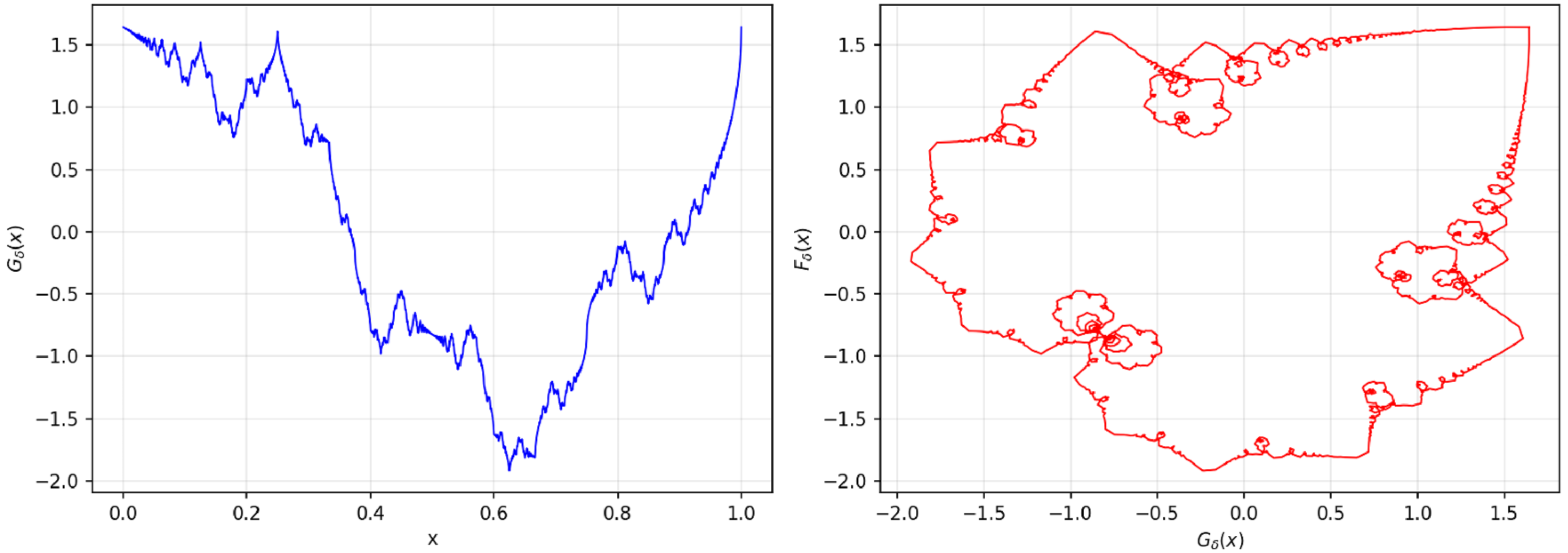}
\caption{$graph(G_{\delta})$ (in blue)
and $graph\big((G_{\delta}, F_{\delta})\big)$ (in red) with $\delta=1$, $g(x)=\sin(2\pi x)+\cos(2\pi x)$
and $f(x)=-\sin(2\pi x)+\cos(2\pi x)$.}
\end{figure}

\begin{figure} \label{Graph6}
\centering
\includegraphics[width=6.5 in, height=3.1 in]{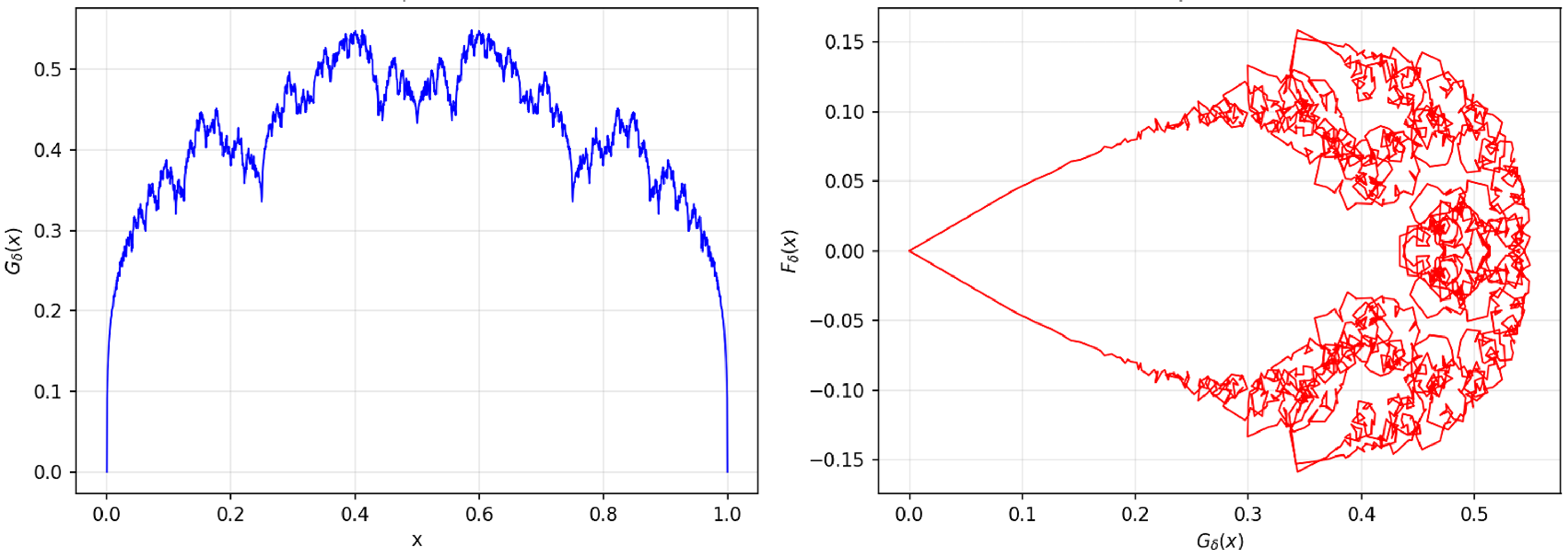}
\caption{$graph(G_{\delta})$ (in blue)
and $graph\big((G_{\delta}, F_{\delta})\big)$ (in red) with $\delta=\frac{1}{2}$, $g(x)=\{x\}(1-\{x\})
=\frac{1}{6}-\frac{1}{\pi^2}\sum\limits_{n=1}^{\infty}\frac{\cos(2\pi nx)}{n^2}$ and
$f(x)=\frac{1}{\pi^2}\sum\limits_{n=1}^{\infty}\frac{\sin(2\pi nx)}{n^2}$.}
\end{figure}


\begin{thebibliography}{99}
\bibitem{Apostol1976}
        Apostol, T. M.: Introduction to Analytic Number Theory. Springer-Verlag, New York, (1976)

\bibitem{Chamizo2004}
       Chamizo, F.: Automorphic forms and differentiability properties. Trans. Amer. Math. Soc.
{\bf 356}(5), 1909-1935 (2004)

\bibitem{Chamizo1999} Chamizo, F., C\'ordoba, A.: Differentiability and dimension of some fractal Fourier series.
                      Adv. Math. {\bf 142}(2), 335-354 (1999)

\bibitem{Cordoba2008} C\'ordoba, A.: Encounters at the interface between number theory and harmonic analysis,
       in: Proceedings of the "Segundas Jornadas de Teor\'ia de N\'umeros", Bibl. Rev. Mat. Iberoamericana, Madrid, 101-118 (2008)

\bibitem{Eceizabarrena2020} Eceizabarrena, D.: Geometric differentiability of Riemann's non-differentiable function.
                             Adv. Math. {\bf 366}, 107091, 39 pp. (2020)

\bibitem{Eceizabarrena2021} Eceizabarrena, D.: On the Hausdorff dimension of Riemann's non-differentiable function.
                             Trans. Amer. Math. Soc. {\bf 374}(11), 7769-7713 (2021)

\bibitem{Falconer2014}  Falconer, K. J.: Fractal Geometry: Mathematical Foundations and Applications.
                        3rd ed., John Wiley \& Sons, Chichester, UK, (2014)

\bibitem{Gerver1970} Gerver, J.: The differentiability of the Riemann function at certain rational multiples of $\pi$.
                     Amer. J. Math. {\bf 92}, 33-55 (1970)

\bibitem{Grafakos2014} Grafakos, L.: Classical Fourier Analysis. 3rd ed., Springer, New York, (2014)

\bibitem{Hadamard1896} Hadamard, J.: Sur la distribution des z\'eros de la fonction et ses cons\'equences arithm\'etiques.
Bulletin de la Soci\'et\'e Math\'ematique de France, {\bf 24}, 199-220 (1896)

\bibitem{Hardy1916} Hardy, G. H.: Weierstrass's non-differentiable function. Trans. Amer. Math. Soc.
                    {\bf 17}, 301-325 (1916)

\bibitem{HardyWright2008}
        Hardy, G. H., Wright, E. M.: An Introduction to the Theory of Numbers. 6th ed., Oxford University Press, (2008)

\bibitem{Hindry2011} Hindry, M.: Arithmetics. Springer-Verlag, (2011)

\bibitem{Hunt1998}
         Hunt, B. R.: The Hausdorff Dimension of Graphs of Weierstrass Functions, Proc. Amer. Math. Soc. {\bf 126}(3), 791-800 (1998)

\bibitem{Ireland-Rosen1990}
        Ireland, K., Rosen, M.: A Classical Introduction to Modern Number Theory. 2nd ed., Springer-Verlag, (1990)

\bibitem{Jaffard1996}
         Jaffard, S.: The spectrum of singularities of Riemann's function. Rev. Mat. Iberoamericana \textbf{12}(2), 441-460 (1996)

\bibitem{Montgomery1994} Montgomery, H. L.: Ten lectures on the interface between analytic number theory and harmonic analysis.
         CBMS Regional Conference Series in Mathematics, 84. Published for the CBMS, Washington, DC; by the AMS, Providence, RI, (1994)

\bibitem{Montgomery} Montgomery, H. L., Vaughan, R. C.: Multiplicative Number Theory I: Classical Theory.
                     Cambridge University Press, Cambridge (2007)

\bibitem{Shen-Ren2021} Ren, H., Shen, W.: A dichotomy for the Weierstrass-type functions.
                       Invent. Math., {\bf 226}(3), 1057-1100 (2021)

\bibitem{Rosen1984} Rosen, K.: Elementary Number Theory and Its Applications. 1st ed., Reading, Mass: Addison-Wesley (1984)

\bibitem{Shen2018} Shen, W.: Hausdorff dimension of the graphs of the classical Weierstrass functions.
                   Math. Z. {\bf289}(1-2), 223-266 (2018)

\bibitem{Stein1993a}Stein, E. M.: Harmonic Analysis: Real Variable Methods, Orthogonality and Oscillatory Integrals.
                    Princeton University Press, Princeton, (1993)


\end{thebibliography}
\end{document}